\documentclass{amsart}
\usepackage{amssymb}
\usepackage{amsbsy, amsthm, amsmath,amstext, amsopn}
\usepackage{tikz-cd}
\usepackage{multicol}
\usepackage{tikz}
\usepackage{mathtools}
\usetikzlibrary{arrows,snakes,backgrounds}
\usetikzlibrary{decorations.markings}
\usetikzlibrary{matrix}
\usetikzlibrary{graphs}

\theoremstyle{definition}
\newtheorem{theorem}{Theorem}[section]
\newtheorem{prop}[theorem]{Proposition}

\newtheorem{cor}[theorem]{Corollary}

\theoremstyle{definition}
\newtheorem{conj}[theorem]{Conjecture}
\newtheorem{definition}[theorem]{Definition}

\newtheorem{observation}[theorem]{Observation}

\theoremstyle{remark}
\newtheorem{rmk}[theorem]{Remark}

\newcommand{\Z}{\mathbb{Z}}

\newcommand{\A}{\mathcal{A}}
\newcommand{\B}{\mathcal{B}}
\newcommand{\X}{\mathcal{X}}

\newcommand{\Spec}{\operatorname{Spec}}
\newcommand{\Conf}{\operatorname{Conf}}

\DeclareMathOperator{\Out}{Out}
\DeclareMathOperator{\Aut}{Aut}
\DeclareMathOperator{\Inn}{Inn}

\newcommand\ontop[2]{\genfrac{}{}{0pt}{}{#1}{#2}}

\newcommand{\dud}[2]{{\small\begin{pmatrix}  #1 \\   \\ #2 \end{pmatrix}}}

\newcommand{\tcfl}[3]{{\small\begin{pmatrix}  & #1 \\ #2 &  \\ & #3 \end{pmatrix}}}
\newcommand{\tcfr}[3]{{\small\begin{pmatrix}  #1 & \\ & #3  \\ #2 & \end{pmatrix}}}
\newcommand{\tcfu}[3]{{\small\begin{pmatrix}  & #1 & \\ #2 & & #3  \\ & & \end{pmatrix}}}
\newcommand{\tcfd}[3]{{\small\begin{pmatrix}  & & \\ #1 & & #3  \\ & #2 & \end{pmatrix}}}

\newcommand{\qcf}[4]{{\small\begin{pmatrix}  & #1 & \\ #2 & & #4  \\ & #3 & \end{pmatrix}}}

\newcommand{\dlr}[2]{{\small\begin{pmatrix}  &  & \\ #1 & & #2  \\ & & \end{pmatrix}}}

\newcounter{Qcount}
\begin{document}

\stepcounter{Qcount}

\title{An Approach to Higher Teichmuller Spaces for Classical Groups}

\author{Ian Le}
\address{Perimeter Institute for Theoretical Physics\\ Waterloo, ON N2L 2Y5}
\email{ile@perimeterinstitute.ca}

\begin{abstract} Let $S$ be a surface, $G$ a simply-connected semi-simple group, and $G'$ the associated adjoint form of the group. In \cite{FG1}, the authors show that the moduli spaces of framed local systems $\X_{G',S}$ and $\A_{G,S}$ have the structure of cluster varieties when $G$ had type $A$. This was extended to classical groups in \cite{Le}. In this paper we give a method for constructing the cluster structure for general reductive groups $G$. The method depends on being able to carry out some explicit computations, and depends on some mild hypotheses, which we state, and which we believe hold in general. These hypotheses hold when $G$ has type $G_2,$ and therefore we are able to construct the cluster structure in this case. We also illustrate our approach by rederiving the cluster structure for $G$ of type $A$. Our goals are to give some heuristics for the approach taken in \cite{Le}, point out the difficulties that arise for more general groups, and to record some useful calculations. Forthcoming work by Goncharov and Shen gives a different approach to constructing the cluster structure on $\X_{G',S}$ and $\A_{G,S}$. We hope that some of the ideas here complement their more comprehensive work.
\end{abstract}

\maketitle

\tableofcontents

\section{Introduction}

Let $S$ be a topological surface $S$ with non-empty boundary, let $G$ be a simply connected, split, semi-simple group, and let $G'$ be the adjoint form of this group. One can then consider $\mathcal{L}_{G,S}$, the moduli space of $G$-local systems on the surface $S$, which are given by representations of $\pi_1(S)$ into $G$. We can consider two variations of this space, $\X_{G',S}$ and $\A_{G,S}$, which were constructed by Fock and Goncharov \cite{FG1}. These spaces are moduli of local systems with the additional data of framing at the boundary of $S$.

Once we add this framing data, the spaces $\X_{G',S}$ and $\A_{G,S}$ become rational, and moreover have positive rational coordinate charts. In many cases it is known that these rational coordinate charts come from a \emph{cluster structure} on these spaces. In these situations, the pair $(\X_{G',S}, \A_{G,S})$ forms a {\it cluster ensemble}. This was shown for $G$ of type $A$ in \cite{FG1} and extended to types $B, C, D$ in \cite{Le}. In this paper we give an approach that we expect to work for all semi-simple groups, carrying out the approach in full when $G$ has type $G_2$, and also showing how it can be used to rederive the cluster structure in type $A$. (A more comprehensive approach can be found in upcoming work of Goncharov and Shen.)

The existence of these cluster structures on $\X_{G',S}$ and $\A_{G,S}$ means that both these spaces have an atlas of coordinate charts such that all transition functions involve only addition, multiplication and division. In other words, these spaces each have a {\it positive atlas} and may be called {\it positive varieties}. Thus these spaces exhibit the phenomenon of total positivity discovered by Lusztig \cite{Lu}. These positive structures coincide with those discovered in \cite{FG1}, and can be used to derive results in higher Teichmuller theory.

In this paper, we give a method for constructing cluster ensemble structures on $\X_{G',S}$ and $\A_{G,S}$ for any semi-simple group $G$. The success of this algorithm depends on calculating that the functions we construct lie in a single tensor invariant space
$$[V_{\lambda} \otimes V_{\mu} \otimes V_{\nu}]^G.$$
Here $V_{\lambda}$ is the irreducible representation of $G$ with highest weight $\lambda$. In \cite{Le}, this computation was carried out for $G$ of type $B, C$ or $D$. In this paper, we carry out these computations in type $A$ and $G_2$, thus extending the results of \cite{Le}.

An important building block for the cluster structures on $\X_{G',S}$ and $\A_{G,S}$ are the cluster structures on double Bruhat cells and in flag varieties constructed in \cite{BFZ}, and based on earlier work \cite{FZ}, \cite{BZ}. The first step consists of constructing a cluster structure on $\X_{G',S}$ and $\A_{G,S}$ in the case where $S$ is an ideal triangle. In this case, the cluster structure is close to the cluster structure on the Borel $B$, which is the same as the cluster structure on the double Bruhat cell $G^{w_0,e}$ for the group $G$. Thus to complete the first step, we must understand the modifications to the cluster structure on $G^{w_0,e}$ we need to make in order to handle the case of a triangle. A cluster structure on each triangle is determined by the following data: an ordering of the vertices in each triangle and a reduced word for the longest element $w_0$ of the Weyl group $W$ of $G$.

In the second step, we first take an ideal triangulations of $S$. We then use a procedure called \emph{amalgamation} (\cite{FG3}), which tells us how to glue the cluster structure on each of these triangles together to get a cluster structure on $\X_{G',S}$ and $\A_{G,S}$. Amalgamating these cluster structures over all the triangles in the triangulation gives us the cluster structure on $\X_{G',S}$ and $\A_{G,S}$. 

In the case that $G=SL_n$ or $PGL_n$, if we take any triangulation and a particular choice of the reduced word for $w_0$, the resulting cluster structure exhibits $S_3$ symmetry (recall that the cluster structure on a triangle initially depended on an ordering of the vertices of the triangle). For other groups, we need to find sequences of mutations that relate the seeds that differ by $S_3$ symmetries. Moreover, in order to relate seeds that correspond to different triangulations, we need to relate seeds coming from triangulations that differ by a ``flip.'' The operation of a flip takes a diagonal of a quadrilateral and replaces by the other diagonal.

Once one shows that $S_3$ symmetries and flips are realized by cluster mutation, we get that mapping class group of $S$ acts by automorphisms of the cluster algebra. We can then say that the mapping class group is contained in the \emph{cluster modular group}. We know this when $G$ has type $A, B, C$ or $D$. In this paper, we give these sequences in the case where $G$ has type $G_2$. We would like to point out that understanding the sequence of mutations realizing $S_3$ symmetries and flips is one of the major remaining difficulties in types $E$ and $F$, in addition to identifying the functions in the cluster algebra as tensor invariants.

While writing this paper, we learned of forthcoming work of Goncharov and Shen that shows that $\X_{G',S}$ and $\A_{G,S}$ are cluster varieties for all reductive groups $G$. Nevertheless, we believe that our viewpoint, and some of the auxiliary results and calculations, are interesting in their own right.

Our hope is that in describing the procedure for constructing the cluster structure on $\A_{G,S}$ in a uniform way, we give some motivations for the constructions in \cite{Le} that seemingly proceed on a case-by-case basis. Furthermore, the examples and computations in this paper can serve as an introduction to the more involved computations in that paper. Here are the aims of this paper:

\begin{enumerate}
\item To give a uniform way to construct the cluster algebra on $\A_{G,S}$ for reductive $G$.
\item To explain the heuristics we used in deriving the cluster structure for classical groups in \cite{Le}.
\item To carry out the procedure for the construction of the cluster algebra in full in types $A_n$ and $G_2$. 
\item To give one point of view on the elements of the cluster mapping class group coming from outer automorphisms of $G$.
\item To record some computations in types $A_n$ and $G_2$ analogous to those performed in \cite{Le} for types $B, C, D$.
\item To give a more conceptual derivation of the sequence of mutations relating the two reduced words for $w_0 \in G$ when $G$ has type $G_2$. This calculation was first performed in \cite{BZ} and again in \cite{FG3}, and in both cases, the derivation was rather complicated.
\end{enumerate}

Here is the outline of the paper. In Section 2, we review some of the work of Fock and Goncharov. We define the spaces $\X_{G',S}$ and $\A_{G,S}$ and relate them to spaces of configurations of points in $G/B$ and $G/U$: $\Conf_m \A_G$ and $\Conf_m \B_G$. We also recall the necessary facts about cluster algebras.

In Section 3, we show how, starting with the constructions of \cite{BFZ}, one can derive the cluster structure on $\Conf_3 \A_G$, assuming that the cluster functions lie in tensor invariant spaces. We also describe the data needed to glue triangles. Throughout this section, our running example will be type $A$. In Section 4, we explain how this construction allows one to realize the action of outer automorphisms of $G$ on $\A_{G,S}$. In Section 5, we verify that the cluster functions lie in tensor invariant spaces in type $G_2$, thus allowing us to construct the cluster algebra structure in this case. We also exhibit the sequences of mutations realizing $S_3$ symmetries and the flip.

\section{Background}

\subsection{Setup}

Let $S$ be an oriented surface. $S$ should be of finite type and have non-empty boundary. The boundary components of $S$ may contain a finite number of marked points. We will always take $S$ to be hyperbolic, meaning that it admits the structure of a hyperbolic surface such that the boundary components that do no contain marked points are cusps, and all the marked points are also cusps. Such $S$ will have ideal triangulations. In practice, this means that $S$ either has negative Euler characteristic, or contains enough marked points on the boundary.

Let $G$ be a split semi-simple algebraic group. When $G$ is adjoint, i.e., has trivial center (for example, when $G=PGL_m$), we can define a higher Teichmuller space $\X_{G,S}$. When $G$ is simply-connected  (for example, when $G=SL_m$), we can define the higher Teichmuller space $\A_{G,S}$. Note that as far as the author knows, the theory of cluster algebras on moduli of local systems only exists for split reductive groups. Splitness seems to be necessary because of the relationship with total positivity. In this paper, we deal only with semi-simple groups because we use the constructions and results of \cite{BFZ}.

The spaces $\X_{G,S}$ and $\A_{G,S}$ will parameterize local systems on $S$ with structure group $G$ with some extra structure of a framing on the boundary components of $S$.

When $S$ has at least one hole, the spaces $\X_{G,S}$ and $\A_{G,S}$ have a distinguished collection of coordinate systems which allows us to extend Lusztig's theory of total positivity to these spaces \cite{Lu}. Our goal will be to show that these coordinate systems are a consequence of the cluster structure on these spaces.

\subsection{Definition of the spaces $\X_{G,S}$ and $\A_{G,S}$}

The data of a framing of a local system involves the flag variety associated to $G$. Choose a Borel subgroup $B \subset G$ (a maximal solvable subgroup) and let $U := [B,B]$ be a maximal unipotent subgroup in $G$. Let ${\mathcal B} := G/B$ be the flag variety. Let $\A := G/U$ be the ``principal affine space.'' We will refer to elements of $\A$ as ``principal flags.''

Let ${\mathcal L}$ be a $G$-local system on $S$. For any space $X$ equipped with a $G$-action, we can form the associated bundle ${\mathcal L}_{X}$. For $X=G/B$ we get the associated flag bundle ${\mathcal L}_{\mathcal B}$, and for $X=G/U$, we get the associated principal flag bundle ${\mathcal L}_{\mathcal A}$.

\begin{definition}
A {\rm framed  $G$-local system on $S$} is a  $G$-local system ${\mathcal L}$ on $S$, and a flat section of the restriction of ${\mathcal L}_{\mathcal B}$ to the punctured boundary of $S$.

The space ${\mathcal X}_{G', S}$ is the moduli space of framed $G$-local systems on $S$. 
\end{definition}

The definition of the space ${\mathcal A}_{G, S}$ is slightly more complicated. It involves twisted local systems. We shall define this notion.

Let $G$ is simply-connected. The maximal length element $w_0$ of the Weyl group of $G$ has two natural lifts to $G$, typically denoted $\overline w_0$ and $\overline{\overline{w_0}}$. Let $s_G:= {\overline w}_0^2 = \overline{\overline{w_0}}^2$. It turns out that $s_G$ is in the center of $G$ and that $s^2_G =e$. Depending on $G$, $s_G$ will have order one or order two. For example, for $SL_n$, $s_G$ has order one or two depending on whether $n$ is odd or even, respectively. For $G=G_2$, $s_G = e$.

The fundamental group $\pi_1(S)$ has a natural central extension by $\Z / 2\Z$: Let $T'S$ be the \emph{punctured tangent bundle}, which is the tangent bundle of $S$ with the zero-section removed. $\pi_1(T'S)$ is a central extension of $\pi_1(S)$ by $\Z$:

$$\Z \rightarrow \pi_1(T'S) \rightarrow \pi_1(S).$$

The quotient of $\pi_1(S)$ by the central subgroup $2 \Z \subset \Z$, gives ${\overline \pi}_1(S)$ which is a central extension of $\pi_1(S)$ by $\Z / 2\Z$:

$$\Z / 2\Z \rightarrow {\overline \pi}_1(S) \rightarrow \pi_1(S).$$

Let $\sigma_S \in {\overline \pi}_1(S)$ denote the non-trivial element of the center.

A {\it twisted $G$-local system} is a representation of ${\overline \pi}_1(S)$ in $G$ such that $\sigma_S$ maps to $s_G$. Such a representation gives a local system on $T'S$.

We now describe the framing data for a twisted local system. A twisted local system $\mathcal{L}$ gives an associated principal affine bundle ${\overline {\mathcal L}}_{\mathcal A}$ on the punctured tangent bundle $T'S$. For the boundary components of $S$, we can construct sections of the punctured tangent bundle lying above these boundary components. For any boundary component, take the section given by the outward pointing unit tangent vectors along this component. We get a bunch of loops and arcs in $T'S$ lie over the boundary of $S$ which we will call the {\em lifted boundary}.

\begin{definition}

A {\it decorated $G$-local system} on $S$ consists of $({\mathcal L}, \alpha)$, where ${\mathcal L}$ is a twisted local system on $S$ and $\alpha$ is a flat section of ${\overline {\mathcal L}}_{\mathcal A}$ restricted to the lifted boundary.

The space ${\mathcal A}_{G, S}$ is the moduli space of decorated $G$-local systems on $S$.
\end{definition}

Note that in the case where $s_G=e$, a decorated local system is just a local system on $S$ along with a flat section of ${\mathcal L}_{\mathcal A}$ restricted to the boundary. One can generally pretend that this is the case without much danger.

\subsection{Relation to configurations of flags}

Let us outline the general procedure for producing positive coordinate systems or cluster charts on the spaces ${\X}_{G', S}$ and ${\A}_{G, S}$. The key idea is the following. If $S' \subset S$ is a surface, then any framed (twisted) local system on $S$ restricts to give a framed (respectively, decorated) local system on $S'$. Our approach will be to take an ideal triangulation of $S'$ and use this to reduce our problem to studying framed (respectively, decorated) local systems on a disc with three or four marked points.

We begin with the case where $S$ is a disc with marked points on the boundary. Let us fix $S$ to be a disc with $m$ marked points on the boundary. In this case, ${\X}_{G', S}$ is the space of configurations of points in the flag variety ${\B}: = G'/B$,
$${\Conf}_{m}({\B}) = G' \backslash (G'/B)^m.$$
We call ${\A}_{G, S}$ the space of twisted configurations of points of the principal affine variety ${\A}:= G/U$. It can be identified with
$${\Conf}_{m}({\A}) = G \backslash (G/U)^m,$$
although this identification is not completely natural. Let us elaborate. Recall that ${\A}_{G, S}$ is the space of twisted local systems on $S$, which consists of local systems on $T'S$ such that the central element is mapped to $s_G$. The framing consists of a flag section of ${\overline {\mathcal L}}_{\mathcal A}$ over the lifted boundary, which consists of outward-pointing tangent vectors along the boundary. Label regions between the $m$ marked points on $S$ by $1, 2, \dots, m$ going counterclockwise. Choose a point in the lifted boundary for each these regions, and call them $x_1, \dots, x_m$. To identify ${\A}_{G, S}$ with a configuration of points in ${\A}:= G/U$, we need to parallel-transport all the sections to a particular point. Let us choose $x_1$. Then there is a unique path in $T'S$ from each $x_i$ to $x_1$ where the tangent vectors rotate clockwise. Parallel transport along these parths allows us to identify ${\A}_{G, S}$ with ${\Conf}_{m}({\A})$.

Suppose that a point in ${\Conf}_{m}({\A})$ becomes identified with $(F_1, \dots, F_m) \in \Conf_m (\A)$. If we had chosen the point $x_2$ and performed the same procedure, we would have gotten the configuration $(F_2, \dots, F_m, s_G\cdot F_1$. This is why the natural cyclic shift map (which we will sometimes call the twisted cyclic shift map) is defined by the formula
$$T \cdot (F_1, F_2, \dots, F_m) = (F_2, \dots, F_m, s_G \cdot F_1).$$

Analogously, on the space ${\Conf}_{m}({\B})$ we can define the cyclic shift map 
$$T: {\Conf}_{m}({\B}) \rightarrow {\Conf}_{m}({\B})$$
by the formula
$$T \cdot (G_1, G_2, \dots, G_m) = (G_2, \dots, G_m, G_1).$$

Once we have the sequence of mutations changes of triangulation, we will see that twisted cyclic shift maps can be realized by elements of the cluster mapping class group.

Let us return to the outlining the contruction of the cluster structure on ${\A}_{G, S}$ which we identify with $\Conf_{m}(\A)$. We get coordinate charts on these spaces attached to any triangulation of an $m$-gon. Fix such a triangulation.

We begin with the space ${\Conf}_{m}({\B})$. If we place the $m$ flags at the vertices of an $m$-gon, then to each triangle in the triangulation of the $m$-gon, we get a configuration of three flags. To this configuration of three flags, we attach some {\em face functions}. The face functions give a positive coordinate chart on ${\Conf}_{3}({\B})$.

Any edge in the triangulation belongs to two triangles. We attach to each edge a set of {\em edge functions}, which depend on the four flags at the corners of these two triangles. For any edge, its edge functions along with the face functions for the triangles sharing that edge form a positive coordinate chart on ${\Conf}_{4}({\B})$. Thus the face functions give invariants of three flags, while the edge functions tell us how to glue two configurations of three flags into a configuration of four flags. Gluing along all edges of a triangulation will give a configuration of $m$ flags.

There are also positive coordinate charts on ${\Conf}_{m}({\A})$ attached to each triangulation of an $m$-gon. Note that the natural order on the set $\{1, \dots, m\}$ induces a natural order on each subset, and in particular for the vertices of any triangle. Therefore, for each triangle, we can identify the cluster structure there with the one on $\Conf_{3}(\A)$. In order to show that the functions we construct on ${\Conf}_{m}({\A})$ are compatible with the twisted cyclic shift map, it will be enough to show this for $\Conf_{3}(\A)$. We will return to this when we discuss $S_3$ symmetries.

In more detail, to each edge in the triangulation, we attach a set of {\em edge functions} which depend on the two flags at the ends of the edge. To each triangle in the triangulation, we get a configuration of three principal flags, and in addition to the edge functions of each pair of edges in the triangle, we attach some {\em face functions}, which depend on all three of the principal flags. In contrast to the situation for ${\Conf}_{m}({\B})$, instead of parameterizing ways of gluing configurations of three flags into a configuration of four flags, the edge functions place restrictions on which triangles can be glued together. Two triangles can be glued along an edge if and only if the edge functions coincide.

We now discuss the general case. Let $S$ be a hyperbolic surface. Take any hyperbolic structure on $S$ such that the boundary components that do no contain marked points are cusps, and all the marked points on the remainng boundary components are also cusps. Then the universal cover of $S$ will be a subset of the hyperbolic plane, and all thesepreimages of these cusps will lie at the boundary at infinity of the hyperbolic plane. An {\em ideal triangulation} of $S$ consists of a triangulation of $S$ that has vertices at the cusps of $S$. Because the edge and face functions for the spaces ${\X}_{G', S}$ and ${\A}_{G, S}$ were local in nature (they only depended on between two and four flags nearby to the edge or face), these functions still make sense for the surface $S$. Taking edge and face functions for the resulting triangulation gives the appropriate coordinate charts on ${\X}_{G', S}$ and ${\A}_{G, S}$. Note that boundary edges give functions on ${\A}_{G, S}$ but not on ${\X}_{G', S}$. 

If we started with positive/cluster charts for ${\Conf}_{3}({\B})$, ${\Conf}_{4}({\B})$, ${\Conf}_{2}({\A})$ and ${\Conf}_{3}({\A})$, we obtain positive/cluster coordinate charts for ${\X}_{G', S}$ and ${\A}_{G, S}$. The coordinate charts coming from different triangulations of the surface give a positive atlas on ${\X}_{G', S}$ and ${\A}_{G, S}$.

One of the goals of this paper will be to show that these edge and face functions can be realized as part of a cluster ensemble structure on the pair of spaces $({\mathcal X}_{G', S}, {\mathcal A}_{G, S})$.

\subsection{Cluster algebras} \label{cluster}

We review here the basic definitions of cluster algebras, following \cite{W}. Cluster algebras are commutative rings that come equipped with a collection of distinguished sets of generators, called \emph{cluster variables}. These are sometimes called $\A$-coordinates, as they are functions on the $\A$ space. Each set of generators forms a \emph{cluster}. Starting from an initial cluster, one generates other clusters by the process of \emph{mutation}.

Each set of generators belongs to a seed, which roughly consists of the set of generators along with a $B$-matrix, which encodes how one mutates from one seed to any adjacent seed.

The same combinatorial data underlying a seed gives rise to a second, related, algebraic structure, called $\X$-coordinates. The $\X$-coordinates are functions on the $\X$ space. They also are grouped into clusters which can be obtained from one another by mutation.

Cluster algebras and $\X$-coordinates are defined by seeds. A seed $\Sigma = (I,I_0,B,d)$ consists of the following data: 
\begin{enumerate}
\item An index set $I$ with a subset $I_0 \subset I$ of ``frozen'' indices. 
\item A rational $I \times I$ \emph{exchange matrix} $B$. It should have the property that $b_{ij} \in \Z$ unless both $i$ and $j$ are frozen.  
\item A set $d = \{d_i \}_{i \in I}$ of positive integers that skew-symmetrize $B$: $$b_{ij}d_j = -b_{ji}d_i$$ for all $i,j \in I.$ The integers $d_i$ are called \emph{multipliers}.
\end{enumerate}

For most purposes, the values of $d_i$ are only important up to simultaneous scaling. Also note that the values of $b_{ij}$ where $i$ and $j$ are both frozen will play no role in the cluster algebra, though it is sometimes convenient to assign values to $b_{ij}$ for bookkeeping purposes. These values become important in \emph{amalgamation}, where one unfreezes some of the frozen variables.

Let $k \in I \setminus I_0$ be an unfrozen index of a seed $\Sigma$. Then we can perform a mutation at $k$ to get another seed $\Sigma' = \mu_k(\Sigma)$. The frozen variables and $d_i$ are preserved, and the exchange matrix $B'$ of $\Sigma'$ satisfies
\begin{align}\label{eq:matmut}
b'_{ij} = \begin{cases}
-b_{ij} & i = k \text{ or } j=k \\
b_{ij} & b_{ik}b_{kj} \leq 0 \\
b_{ij} + |b_{ik}|b_{kj} & b_{ik}b_{kj} > 0.
\end{cases}
\end{align}

To a seed $\Sigma$ we associate a collection of \emph{cluster variables} $\{A_i\}_{i \in I}$ and a split algebraic torus $\A_\Sigma := \Spec \Z[A^{\pm 1}_I]$.

If $\Sigma'$ is obtained from $\Sigma$ by mutation at $k \in I \setminus I_0$, there is a birational \emph{cluster transformation} $\mu_k: \A_\Sigma \to \A_{\Sigma'}$.  This map is given by the \emph{exchange relation}

\begin{align}\label{eq:Atrans}
\mu_k^*(A'_i) = \begin{cases}
A_i & i \neq k \\
A_k^{-1}\biggl(\prod_{b_{kj}>0}A_j^{b_{kj}} + \prod_{b_{kj}<0}A_j^{-b_{kj}}\biggr) & i = k.
\end{cases}
\end{align} 

Given a seed $\Sigma$ we also associate a second algebraic torus $\X_\Sigma := \Spec \Z[X_I^{\pm 1}]$.  If $\Sigma'$ is obtained from $\Sigma$ by mutation at $k \in I \setminus I_0$, we again have a birational map $\mu_k: \X_\Sigma \to \X_{\Sigma'}$, defined by
\begin{align}\label{eq:Xtrans}
\mu_k^*(X'_i) = \begin{cases}
X_iX_k^{[b_{ik}]_+}(1+X_k)^{-b_{ik}} & i \neq k  \\
X_k^{-1} & i = k,
\end{cases}
\end{align}
where $[b_{ik}]_+:=\mathrm{max}(0,b_{ik})$.

There is a natural map from $\A_{\Sigma}$ to $\X_{\Sigma}$. Let us assume that the entries of the $B$-matrix are all integers. Then we can define $p: \A_\Sigma \to \X_\Sigma$ by
$$p^*(X_i) = \prod_{j \in I}A_j^{B_{ij}}.$$

This formula appears to depend on the seed, but it actually intertwines the mutation of both the $\A$-coordinates and the $\X$-coordinates. We have the commutative diagram

\vspace{-2mm}
 \[
\begin{tikzcd}
\A_{\Sigma} \arrow[dashed]{r}{\mu_k} \arrow{d}{p} & \A_{\Sigma'} \arrow{d}{p'} \\
\X_{\Sigma} \arrow[dashed]{r}{\mu_k} & \X_{\Sigma'} \\
\end{tikzcd}
\]

\section{General groups}
\subsection{The cluster algebra on $B^-$}

We describe here how one can find the cluster algebra structure on $\Conf_n \A_G$ for any simply-connected simple group $G$. In particular, we believe this procedure will give the correct seed for $\Conf_n \A_G$ when $G$ is of type $E, F$ or $G$. The procedure relies on verifying some facts about certain functions on $\Conf_n \A_G$ through a non-obvious computation, which we will only carry out the process in full for type $G_2$. 

We first give a procedure for constructing the cluster structure on $\Conf_3 \A_G$. In \cite{Le}, for all classical groups $G$, we related the cluster structure on $\Conf_3 \A_{G}$ to Berenstein, Fomin and Zelevinsky's cluster structure on $B$, the Borel in the group $G$ (\cite{BFZ}). Let us review how this story goes. Let $(A_1, A_2, A_3)$ be a triple of principal flags in $\Conf_3 \A_G$. There will be functions attached to these vertices that only depend on two of the three flags. We will call these the {\em edge functions}. Let us call all other functions, which depend on all three flags, {\em face functions}. We wish to construct the face functions, as well as the edge functions attached to the edges $A_1A_2$, $A_2A_3$, and $A_1A_3$.

We first consider the subset of $\Conf_3 \A_G$ given by triples of principal flags in the image of $B^-$ under the map
$$i: b \in B^- \rightarrow (U^-, \overline{w_0}U^-, b\cdot \overline{w_0}U^-) \in \Conf_3 \A_{G}.$$

The cluster structure on $B^-$ will give us the cluster structure on $\Conf_3 \A_G$ with all the functions on the edge $A_1A_2$ removed. This will give us all the face functions, as well as all the functions on the two edges $A_2A_3$ and $A_1A_3$. It also gives us all the arrows of the quiver connecting face vertices as well as face vertices with the vertices on edges $A_2A_3$ and $A_1A_3$.

Let us review how to construct the cluster structure on $B^-$. A more detailed account can be found in \cite{BFZ} and \cite{FG3}.

Let $w_0$ be the longest word in the Weyl group of $G$. The double Bruhat cell 
$$G^{w_0,e}:= B^+w_0B^+ \cap B^-eB^-$$
is the open part of $B^-$. Take any reduced-word $s_{i_1}\dots s_{i_K}$ for $w_0$. Our convention will be that we read the word from right to left, i.e., the simple reflection $s_K$ followed by the simple reflection $s_{K-1}$, etc. Here $1 \leq i_j \leq n$, where $n$ is the number of nodes in the Dynkin diagram. For each node $i$ of the Dynkin diagram, there is a corresponding simple reflection $s_i$ in the Weyl group $W$.

For each reduced-word expression for $w_0$ there is a corresponding seed for the cluster algebra on $B^-$. The $B$-matrix for this seed can encoded via a quiver. This quiver will have $n+K$ vertices, of which $2n$ are frozen edge vertices. Each vertex of the quiver belongs to one of the nodes of the Dynkin diagram. If the simple reflection $s_i$ occurs $a_i$ times in the reduced-word for $w_0$, there will be $a_i+1$ vertices belonging to the node $i$. Of these, two vertices (the first and the last) will be frozen.

The multipliers $d_i$ for the vertices are determined by the nodes they are associated with. If a node in the Dynkin diagram is associated with a short root, the multipliers for the vertices belong to this node are $1$. If a node in the Dynkin diagram is associated with a long root, the multipliers for the vertices belong to this node are $2$ in types $B, C, F$ and $3$ in type $G$. Alternatively, let the node $i$ correspond to the root $\alpha_i$, and let us normalize the lengths of the roots so that the short roots have length $1$. Then the multipliers for the vertices belonging to node $i$ are $(\textrm{length of } \alpha_i)^{2}$.

We have described the vertices of the quiver and the multipliers for these vertices. The arrows in the quiver carry the following information:

\begin{itemize}
\item An arrow from $j$ to $i$ means that $b_{ij}>0$ and $b_{ji}<0$.
\item $|b_{ij}|=2$ if $d_i=2$ and $d_j=1$.
\item $|b_{ij}|=3$ if $d_i=3$ and $d_j=1$.
\item $|b_{ij}|=1$ otherwise.
\end{itemize}

The procedure for constructing this quiver from the reduced word for $w_0$ is complicated, and perhaps unilluminating, to state in full generality. Instead, let us look at the example of $SL_4$ in detail.

One reduced word for $w_0 \in W_{SL_4}$ is 
$$s_1 s_2 s_1 s_3 s_2 s_1.$$
The corresponding quiver is for $B^-$ is

\begin{center}
\begin{tikzpicture}[scale=2]

  \node (x103) at (-0.5,0.9) {\Large $\ontop{x_{103}}{\bullet}$};
  \node (x013) at (0.5,0.9) {\Large $\ontop{x_{013}}{\bullet}$};

  \node (x202) at (-1,0) {\Large $\ontop{x_{202}}{\bullet}$};
  \node (x112) at (0,0) {\Large $\ontop{x_{112}}{\bullet}$};
  \node (x022) at (1,0) {\Large $\ontop{x_{022}}{\bullet}$};

  \node (x301) at (-1.5,-0.9) {\Large $\ontop{x_{301}}{\bullet}$};
  \node (x211) at (-0.5,-0.9) {\Large $\ontop{x_{211}}{\bullet}$};
  \node (x121) at (0.5,-0.9) {\Large $\ontop{x_{121}}{\bullet}$};
  \node (x031) at (1.5,-0.9) {\Large $\ontop{x_{031}}{\bullet}$};


  \draw [->] (x013) to (x103);
  \draw [->] (x022) to (x112);
  \draw [->] (x112) to (x202);
  \draw [->] (x031) to (x121);
  \draw [->] (x121) to (x211);
  \draw [->] (x211) to (x301);

  \draw [->] (x121) to (x022);
  \draw [->] (x211) to (x112);
  \draw [->] (x112) to (x013);
  \draw [->, dashed] (x301) to (x202);
  \draw [->, dashed] (x202) to (x103);

  \draw [->] (x202) to (x211);
  \draw [->] (x103) to (x112);
  \draw [->] (x112) to (x121);
  \draw [->, dashed] (x013) to (x022);
  \draw [->, dashed] (x022) to (x031);

\draw[yshift=-1.5cm]
  node[below,text width=6cm] 
  {
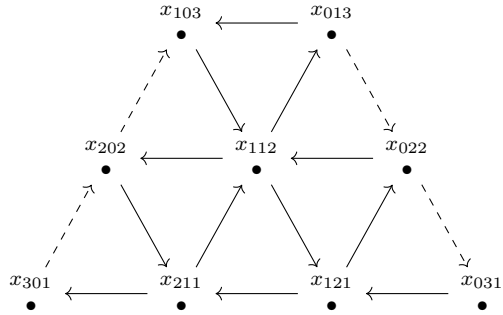
  Figure 1. The quiver for the cluster algebra on $B^-_{SL_{4}}$.
  };

\end{tikzpicture}
\end{center}

For the moment, let us ignore the names of the vertices. In the above quiver, there are three rows. The vertices belonging to node $1$ of the Dynkin diagram are on the bottom row; the vertices belonging to node $2$ are in the middle row; and the vertices belonging to node $3$ are on the top row. $SL_n$ is a simply-laced group, so $d_i=1$ for all vertices. The vertices on either end of each row are frozen vertices. The leftmost vertices in each row are the edge functions for the edge $A_1A_3$. The rightmost vertices in each row are the edge functions for the edge $A_2A_3$. The flags $A_1,A_2, A_3$ are oriented as follows:

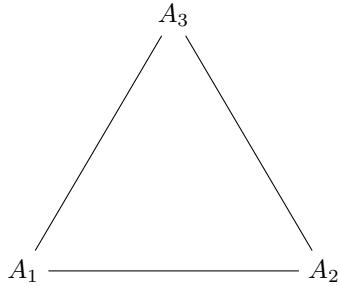
\begin{center}
\begin{tikzpicture}[scale=2]

  \node (A_1) at (-1,0) {$A_1$};
  \node (A_2) at (1,0) {$A_2$};
  \node (A_3) at (0,1.7) {$A_3$};

  \draw [] (A_1) -- (A_2);
  \draw [] (A_2) -- (A_3);
  \draw [] (A_3) -- (A_1);

\end{tikzpicture}
\end{center}

The dotted arrows only go between frozen vertices. A dotted arrow between vertices $i$ and $j$ means that $b_{ij}$ is half what it would be if the arrow were solid. Thus dotted arrows are ``half-arrows.''

Each occurence of the reflections $s_1, s_2, s_3$ gives a portion of the quiver as depicted in Figures 2-4:

\begin{center}
\begin{tikzpicture}[scale=2]

  \node (x211) at (0,0) {$\bullet$};

  \node (x121) at (-0.5,-0.9) {$\bullet$};
  \node (x112) at (0.5,-0.9) {$\bullet$};

  \draw [->, postaction={decorate}] (x112) -- (x121) node [midway, above] {$s_1$};
  \draw [->, dashed] (x121) to (x211);
  \draw [->, dashed] (x211) to (x112);

\draw[yshift=-1cm]
  node[below,text width=6cm] 
  {
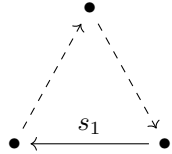
  Figure 2. Portion of the quiver corresponding to the simple reflection $s_1$.
  };

\end{tikzpicture}
\end{center}

\begin{center}
\begin{tikzpicture}[scale=2]

  \node (x310) at (0,0.9) {$\bullet$};

  \node (x220) at (-0.5,0) {$\bullet$};
  \node (x211) at (0.5,0) {$\bullet$};

  \node (x121) at (0,-0.9) {$\bullet$};

  \draw [->, postaction={decorate}] (x211) -- (x220) node [midway, above] {$s_2$};

  \draw [->, dashed] (x121) to (x211);
  \draw [->, dashed] (x220) to (x310);

  \draw [->, dashed] (x220) to (x121);
  \draw [->, dashed] (x310) to (x211);

\draw[yshift=-1cm]
  node[below,text width=6cm] 
  {
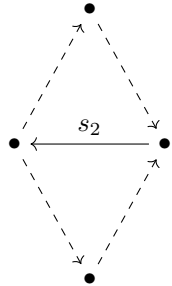
  Figure 3. Portion of the quiver corresponding to the simple reflection $s_2$.
  };

\end{tikzpicture}
\end{center}

\begin{center}
\begin{tikzpicture}[scale=2]

  \node (x310) at (-0.5,0.9) {$\bullet$};
  \node (x301) at (0.5,0.9) {$\bullet$};

  \node (x211) at (0,0) {$\bullet$};

  \draw [->, postaction={decorate}] (x301) -- (x310) node [midway, above] {$s_3$};
  \draw [->,dashed] (x211) to (x301);
  \draw [->,dashed] (x310) to (x211);

\draw[yshift=-0.5cm]
  node[below,text width=6cm] 
  {
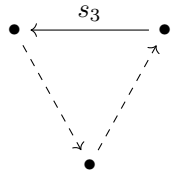
  Figure 4. Portion of the quiver corresponding to the simple reflection $s_3$.
  };

\end{tikzpicture}
\end{center}

We glue these these pieces together according to the reduced word for $w_0$ as in Figure 5 to obtain the quiver for the cluster algebra. The triangles $x_{301}x_{202}x_{211}$, $x_{211}x_{112}x_{121}$, and $x_{121}x_{022}x_{031}$ correspond the simple reflection $s_1$. The quadrilaterals $x_{202}x_{211}x_{112}x_{103}$ and $x_{112}x_{121}x_{022}x_{013}$ correspond the simple reflection $s_1$. The triangle $x_{103}x_{112}x_{013}$ corresponds to the simple reflection $s_3$.

\begin{center}
\begin{tikzpicture}[scale=2]

  \node (x310) at (-0.5,0.9) {$\bullet$};
  \node (x301) at (0.5,0.9) {$\bullet$};

  \node (x220) at (-1,0) {$\bullet$};
  \node (x211) at (0,0) {$\bullet$};
  \node (x202) at (1,0) {$\bullet$};

  \node (x130) at (-1.5,-0.9) {$\bullet$};
  \node (x121) at (-0.5,-0.9) {$\bullet$};
  \node (x112) at (0.5,-0.9) {$\bullet$};
  \node (x103) at (1.5,-0.9) {$\bullet$};


  \draw [->, postaction={decorate}] (x301) -- (x310) node [midway, above] {$s_3$};
  \draw [->, postaction={decorate}] (x202) -- (x211) node [midway, above] {$s_2$};
  \draw [->, postaction={decorate}] (x211) -- (x220) node [midway, above] {$s_2$};
  \draw [->, postaction={decorate}] (x103) -- (x112) node [midway, above] {$s_1$};
  \draw [->, postaction={decorate}] (x112) -- (x121) node [midway, above] {$s_1$};
  \draw [->, postaction={decorate}] (x121) -- (x130) node [midway, above] {$s_1$};

  \draw [->] (x112) to (x202);
  \draw [->] (x121) to (x211);
  \draw [->] (x211) to (x301);
  \draw [->, dashed] (x130) to (x220);
  \draw [->, dashed] (x220) to (x310);

  \draw [->] (x220) to (x121);
  \draw [->] (x310) to (x211);
  \draw [->] (x211) to (x112);
  \draw [->, dashed] (x301) to (x202);
  \draw [->, dashed] (x202) to (x103);

\draw[yshift=-1.5cm]
  node[below,text width=6cm] 
  {
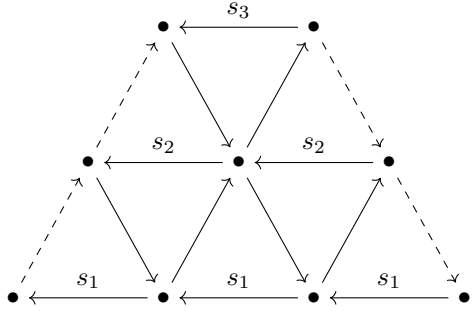
  Figure 5. How the quiver corresponds to the reduced word $s_1 s_2 s_1 s_3 s_2 s_1$ for $B^-_{SL_{4}}$.
  };

\end{tikzpicture}
\end{center}

A dotted arrow is equivalent to half a solid arrow. When we perform the gluing, two dotted arrows in the same direction glue to give us a solid arrow, whereas two dotted arrows in the opposite direction cancel to give us no arrow. (In the $SL_4$ example above, it turns out we never glue dotted arrows going in the opposite direction, but this does happen in general.) This is the process of amalgamation. For general groups, the picture is similar: for a reduced word for $w_0$, each simple reflection gives a portion of the quiver, and amalgamating these portions together gives the full quiver for the cluster algebra on $B^-$.

Let us now describe the functions attached to the vertices of the quiver. The functions are given by \emph{generalized minors} of $B^-$, which we now define. Let $G_0 = U^- H U^+ \subset G$ be the open subset of $G$ consisting of elements that have a Gaussian decomposition $x=[x]_- [x]_0 [x]_+$. Then for any two elements $u, v \in W$, and any fundamental weight $\omega_i$, we have the {\em generalized minor} $\Delta_{u\omega_i, v\omega_i}(x)$ defined by

$$\Delta_{u\omega_i, v\omega_i}(x) := ([\overline{u}^{-1}x \overline{v}]_0)^{\omega_i}.$$
The formula gives a well-defined value when $\overline{u}^{-1}x \overline{v} \in G_0$, but may have poles elsewhere.

We start with a reduced word for $w_0$:
$$w_0=s_{i_1} s_{i_2} s_{i_3} \cdots s_{i_{K-1}} s_{i_K}$$
For $1 \leq l \leq K$, we have the subword $$u_l:=s_{i_1}\dots_{i_l}.$$
In our situation, we are interested in the generalized minors $\Delta_{u\omega_i, v\omega_i}(x)$ when $v=e$ and $u=e$ or $u=u_{l}=s_{i_1} s_{i_2} s_{i_3} \cdots s_{i_l}$ for $1 \leq l \leq K$.

Then the cluster functions on $B^-$ given in \cite{BFZ} are as follows:

\begin{itemize}
\item First take $u, v=e$. For each $i$, $1 \leq i \leq n$, we have the functions $\Delta_{\omega_i, \omega_i}$. 
\item For each value of $l$, $1 \leq l \leq K$, we take$v=e$ and $u=u_{l}$ and get the corresponding function $$\Delta_{u_{l}\omega_{i_l}, \omega_{i_l}}.$$
\end{itemize}

Let us now match the functions with the vertices of the quiver. Suppose that in the reduced word for $w_0$ that we have chosen, there are $a_i$ occurences of the simple reflection $s_i$. Then the quiver contains $a_i+1$ vertices belonging to the node $i$ of the Dynkin diagram. Let us call these vertices $x_{i0}, x_{i1}, x_{i2}, \dots, x_{ia_i}$. (This labelling does not coincide with the labelling of the vertices in our example; the logic behind the labelling in our example will become clear later. This labelling does, however agree with the conventions used in \cite{Le}.)

Then the function $\Delta_{\omega_i, \omega_i}$ will be attached to the vertex $x_{i0}$. If $s_{i_l}$ is the $j^{th}$ occurence of the simple reflection $s_i$ in the reduced word for $w_0$, then the function $\Delta_{u_{l}\omega_{i_l}, \omega_{i_l}},$ will be attached to the vertex $x_{ij}$.

The vertices along the edge $A_1A_3$ will be $x_{i0}$ for $1 \leq i \leq n$, and the associated functions are $\Delta_{\omega_i, \omega_i}$. The vertices along the edge $A_2A_3$ will be $x_{ia_i}$ and the associated functions are also edge functions. All the remaining functions are face functions.

Let us compute these functions in our example where $G=SL_4$. Straightforward computation gives that the functions $\Delta_{u\omega_i, v\omega_i}$ that we are interested in are given by minors with the following rows and columns:

\medskip

\begin{center}
\begin{tabular}{ l  c  c }

  Function & columns & rows \\ \hline
  $\Delta_{\omega_1, \omega_1}$ & 1 & 1 \\
  $\Delta_{s_1\omega_1, \omega_1}$ & 1 & 2 \\
  $\Delta_{s_1 s_2 s_1\omega_1, \omega_1}$ & 1 & 3 \\
  $\Delta_{w_0\omega_1, \omega_1}$ & 1 & 4 \\
  $\Delta_{\omega_2, \omega_2}$ & 1, 2 & 1, 2 \\
  $\Delta_{s_1 s_2\omega_2, \omega_2}$ & 1, 2 & 2, 3 \\
  $\Delta_{s_1 s_2 s_1 s_3 s_2\omega_2, \omega_2}$ & 1, 2 & 3, 4 \\
  $\Delta_{\omega_3, \omega_3}$ & 1, 2, 3 & 1, 2, 3 \\
  $\Delta_{s_1 s_2 s_1 s_3\omega_3, \omega_3}$ & 1, 2, 3 & 2, 3, 4 \\
\end{tabular}
\end{center}

\medskip

The minors correspond to the vertices of the quiver as follows:

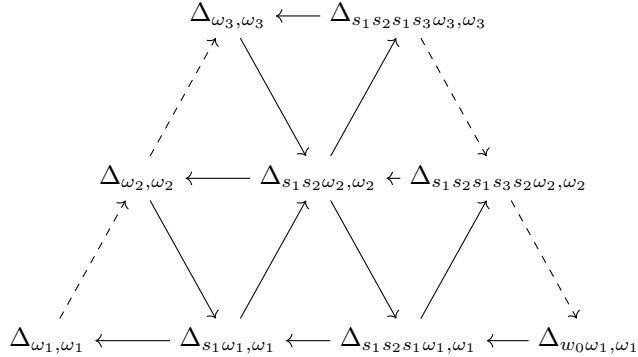
\begin{center}
\begin{tikzpicture}[scale=2.4]

  \node (x103) at (-0.5,0.9) {$\Delta_{\omega_3, \omega_3}$};
  \node (x013) at (0.5,0.9) {$\Delta_{s_1 s_2 s_1 s_3\omega_3, \omega_3}$};

  \node (x202) at (-1,0) { $\Delta_{\omega_2, \omega_2}$};
  \node (x112) at (0,0) { $\Delta_{s_1 s_2\omega_2, \omega_2}$};
  \node (x022) at (1,0) { $\Delta_{s_1 s_2 s_1 s_3 s_2\omega_2, \omega_2}$};

  \node (x301) at (-1.5,-0.9) { $\Delta_{\omega_1, \omega_1}$};
  \node (x211) at (-0.5,-0.9) { $\Delta_{s_1\omega_1, \omega_1}$};
  \node (x121) at (0.5,-0.9) { $\Delta_{s_1 s_2 s_1\omega_1, \omega_1}$};
  \node (x031) at (1.5,-0.9) { $\Delta_{w_0\omega_1, \omega_1}$};

  \draw [->] (x013) to (x103);
  \draw [->] (x022) to (x112);
  \draw [->] (x112) to (x202);
  \draw [->] (x031) to (x121);
  \draw [->] (x121) to (x211);
  \draw [->] (x211) to (x301);

  \draw [->] (x121) to (x022);
  \draw [->] (x211) to (x112);
  \draw [->] (x112) to (x013);
  \draw [->, dashed] (x301) to (x202);
  \draw [->, dashed] (x202) to (x103);

  \draw [->] (x202) to (x211);
  \draw [->] (x103) to (x112);
  \draw [->] (x112) to (x121);
  \draw [->, dashed] (x013) to (x022);
  \draw [->, dashed] (x022) to (x031);

\end{tikzpicture}
\end{center}

More generally, when $G=SL_n$, the Weyl group is generated by simple reflections $s_1, s_2, \dots, s_{n-1}$. The reduced word we will use is 
$$w_0=s_1 s_2 \dots s_{n-1} s_1 s_2 \dots s_{n-2} \dots s_1 s_2 s_3 s_1 s_2 s_1.$$
In this word, there are $n-k$ occurences of the simple reflection $s_k$. The functions $\Delta_{\omega_i, \omega_i}$ will be given by minors consisting of the first $i$ columns and the first $i$ rows. Suppose that $s_{i_l}$ is the $j^{th}$ occurence of the simple reflection $s_i$ in the reduced word for $w_0$. Then the function $\Delta_{u_{l}\omega_{i}, \omega_{i}}$ will given by the minor consisting of the first $i$ columns and rows $j+1, j+2, \dots, j+i$.

\subsection{Extending the cluster structure to $\Conf_3 \A_G$}

The construction of the previous section determines the cluster structure on the triangle $A_1A_2A_3$ with the edge $A_1A_2$ removed. In order to complete the construction we need to do the following:
\begin{enumerate}
\item Construct the functions attached to the edge $A_1A_2$.
\item Determine the multipliers for the functions attached to the edge $A_1A_2$.
\item Determine the arrows in the quiver connecting the functions on the edge $A_1A_2$ to the face functions.
\item Determine the arrows in the quiver connecting the functions on the edge $A_1A_2$ to the other edge functions.
\item  Determine the (possibly fractional) arrows in the quiver connecting the functions on the edge $A_1A_2$ to each other.
\end{enumerate}

Let us explain how to carry out these steps in turn. First, a simple calculation shows that the functions along the edges $A_1A_3$ and $A_2A_3$ as determined above are given by the canonical invariants of the tensor product
$$[V_{\lambda} \otimes V_{\lambda^{\vee}}]^G.$$
where $\lambda$ runs over the fundamental weights $\omega_i$, and $\lambda^{\vee}$ is the weight corresponding to the representation dual to $V_{\lambda},$ in other words, $\lambda^{\vee}=-w_0(\lambda)$. It is then natural to construct the functions along the edge $A_1A_2$ similarly. This determines all the functions for the cluster algebra on $\Conf_3 \A$.
 
Second, the multipliers for all the functions except for those on the edge $A_1A_2$ are determined by the construction of \cite{BFZ}. One only needs to observe that the multiplier attached to the edge function which is given by the invariant in 
$$[V_{\omega_i} \otimes V_{\omega_i^{\vee}}]^G$$
is given by the multiplier attached to the $i$-th node of the Dynkin diagram.

The third step is somewhat more subtle. We now know all the functions for the cluster algebra, and just need to determine the arrows in the quiver. The algebra of functions on $\Conf_3 \A_G$ is naturally graded by dominant weights of $G$. In order to carry out our algorithm for determining the arrows, we need the following conjecture to hold:

\begin{conj} \label{minorweight} Each of the functions $\Delta_{u_{l}\omega_{i_l}, \omega_{i_l}}$ lies in a single graded piece
$$[V_{\lambda} \otimes V_{\mu} \otimes V_{\nu}]^G$$
of $\mathcal{O}(\Conf_3 \A_G)$, for some dominant weights $\lambda, \mu, \nu$.
\end{conj}

There are good reasons to believe this conjecture, among them the Duality Conjectures of Fock and Goncharov, \cite{FG1}. We also know that the conjecture holds in many cases.

\begin{prop} \label{Aminors} The conjecture holds in type $A$.
\end{prop}

\begin{proof} Recall that for $SL_n$, the function $\Delta_{u_{l}\omega_{i}, \omega_{i}}$ was given by the minor consisting of the first $i$ columns and rows $j+1, j+2, \dots, j+i$. This means that it is given by an invariant in the space
$$[V_{\omega_k} \otimes V_{\omega_{j}} \otimes V_{\omega_i}]^G$$
where $k=n-i-j$. If $U^-$ is given by the flag consisting of the forms
$$A_{1k}:=(-1)^ke_{n-k+1} \wedge e_{n-k+2} \wedge \cdots \wedge e_n,$$
and $\overline{w_0}U^-$ is given by  the flag consisting of the forms
$$A_{2j}:=e_1 \wedge e_2 \wedge \cdots \wedge e_j,$$
and the columns of $b \in B^-$ are given by vectors $v_1, \dots, v_n$, so that $b\overline{w_0}U^-$ is given by forms
$$A_{3i}:=v_1 \wedge v_2 \wedge \cdots \wedge v_i,$$
then $\Delta_{u_{l}\omega_{i}, \omega_{i}}$ is given by
$$A_{1k} \wedge A_{2j} \wedge A_{3i}.$$
In our example of $SL_4$, we get that the invariant attached to the vertex $x_{ijk}$ lies in
$$[V_{\omega_i} \otimes V_{\omega_{j}} \otimes V_{\omega_k}]^G.$$
\end{proof} 

\begin{rmk} Although we can calculate that $\Delta_{u_{l}\omega_{i}, \omega_{i}}$ is given by a tensor product invariant in many cases, we don't know of any reason why this is the case. It would be interesting to find a more satisfying explanation of why generalized minors are given by tensor product invariants.
\end{rmk}

\begin{theorem} The conjecture holds in types $B, C, D$.
\end{theorem}

The above theorem was proven in \cite{Le}, where we verified that the functions $\Delta_{u_{l}\omega_{i}, \omega_{i}}$ were given by tensor product invariants. In fact, we identified these invariants using webs. Thus our conjecture is true for all the classical groups. We show later that it holds in type $G_2$, and we expect this to hold in general. In the following constructions, we assume that this is the case.

In order to complete our construction of the arrows between vertices on the edge $A_1A_2$ and face vertices, we need to use the following heuristic: we expect that the space $\Conf_3 \A_G$ has as its corresponding $\X$-variety the space $\Conf_3 \B_G$. Thus, for all the unfrozen vertices (which correspond to face functions), we need that the corresponding $X$-coordinate is a function on $\Conf_3 \B_G$. A rational function on $\Conf_3 \A_G$ descends to $\Conf_3 \B_G$ if and only if it is in the $(0,0,0)$-th graded piece of $\mathcal{O}(\Conf_3 \A_G)$. Because we have that
$$p^*(X_i) = \prod_{j \in I}A_j^{B_{ij}},$$
if $A_j$ lies in the graded piece $(\lambda_j, \mu_j, \nu_j)$, we need
\begin{equation} \label{face} \sum_{j \in I}B_{ij} (\lambda_j, \mu_j, \nu_j) =0
\end{equation}
for all $i \in I$ which is unfrozen. Note that this sum is taking values in $\Lambda^3$, where $\Lambda$ is the weight lattice. It is not difficult to see that condition ~\ref{face} determines uniquely the arrows between the face vertices and the vertices on edge $A_1A_2$. The existence of a choice of the values $B_{ij}$ for these arrows satisfying this condition can be checked on a case-by-case basis. This completes the third step.

Fourth, we need to deal with the (possibly fractional) arrows between vertices on $A_1A_2$ and other edges. There are no $\X$-coordinates attached to the edge vertices. However, we expect that when we glue two triangles together, the edge vertices will then have $\X$-coordinates attached.

Suppose we have a quadrilateral $ABCD$, which is glued from triangles $ABC$ and $ACD$ along edge $AC$:

\begin{center}
\begin{tikzpicture}[scale=2]

  \node (A) at (0,1) {$A$};
  \node (B) at (-1.5,0) {$B$};
  \node (C) at (0,-1) {$C$};
  \node (D) at (1.5,0) {$D$};

  \draw [] (A) -- (B);
  \draw [] (B) -- (C);
  \draw [] (C) -- (D);
  \draw [] (D) -- (A);
  \draw [] (A) -- (C);

\draw[yshift=-1.5cm]
  node[below,text width=6cm] 
  {
  Figure 6. Quadrilateral $ABCD$.
  };

\end{tikzpicture}
\end{center}$$[V_{\omega_i} \otimes V_{\omega_i^{\vee}}]^G$$

Let us consider an edge vertex $i$ on the edge $AC$. Let us suppose that the function attached to vertex $i$ lies in the invariant space
$$[V_{\lambda} \otimes V_{\lambda^{\vee}}]^G,$$
where the grading with respect to the flag at $C$ is $\lambda$ and the grading with respect to the flag at $A$ is $\lambda^{\vee} := -w_0(\lambda)$. Here $\lambda$ will be a fundamental weight for $G$. Suppose that $\lambda=\omega_k$. Each function $A_j$ in the triangle $ABC$ lies in some graded piece $(\lambda_j, \mu_j, \nu_j, 0)$ of $\mathcal{O}(\Conf_4 \A_G)$. Then in the triangle $ABC$, we can again form the sum
$$S_{CA,B,\omega_k}=\sum_{j \in I}B_{ij} (\lambda_j, \mu_j, \nu_j, 0)$$
where we allow fractional values of $B_{ij}$ for arrows between frozen vertices.
Now, we can consider the triangle $ACD$. If from the point of view of triangle $ABC$, vertex $i$ is associated with the fundamental weight $\omega_k$, then from the point of view of triangle $ACD$, the vertex $i$ is associated with the fundamental weight $\omega_k^{\vee}$. Each function $A_j$ in the triangle $ACD$ lies in some graded piece $(\lambda_j, 0, \mu_j, \nu_j)$ of $\mathcal{O}(\Conf_4 \A_G)$. Then in the triangle $ACD$, we can again form the sum
$$S_{AC,D,\omega_k^{\vee}}=\sum_{j \in I}B_{ij} (\lambda_j, 0, \mu_j, \nu_j).$$
Then in order for the vertex $i$ to have an associated $\X$-coordinate, we require that 
\begin{equation}\label{edge0} S_{CA,B,\omega_k}+S_{AC,D,\omega_k^{\vee}}=0.
\end{equation}

This has several implications. First, this means that 
\begin{equation}\label{edge1} S_{AC,D,\omega_k^{\vee}}=(\lambda, 0, \mu, 0)
\end{equation}
for some weights $\lambda$ and $\mu$. This is enough to determine the arrows between vertices on the edge $A_1A_2$ and other edge vertices. This completes the fourth step. The arrows we obtain are integral.

Moreover, we also expect that as we mutate the face vertices in $\Conf_3 \A_G$, the value of $S_{CA,B,\omega_k}$ should not change. This is because mutating face vertices in the triangle $ABC$ leaves the functions in triangle $ACD$ untouched, and for equation ~\ref{edge0} to hold, $S_{CA,B,\omega_k}$ must remain constant. Thus we expect that $\lambda$ and $\mu$ are completely determined by the weight $\omega_k$. In other words, there exist $\lambda_k$ and $\mu_k$ such that in any triangle $ABC$, $S_{AB,C,\omega_k}$ has weight $\lambda_k$ at $A$ and weight $\mu_k$ at $B$. Denote $\omega_{k^*} ;= \omega_k^{\vee}$. Then the equation $S_{CA,B,\omega_k}+S_{AC,D,\omega_k^{\vee}}=0$ means that 
\begin{equation}\label{edge2}\lambda_k+\mu_{k^*}=0=\lambda_{k^*}+\mu_{k}
\end{equation}

We know from the previous steps of our construction the values for $S_{A_3A_1,A_2,\omega_k}$ and $S_{A_2A_3,A_1,\omega_k}$. If we are to be able to use the cluster structure on $\Conf_3 \A_G$ to get cluster structures on $\Conf_m \A_G$, we need $S_{A_3A_1,A_2,\omega_k}=S_{A_2A_3,A_1,\omega_k}$. Our procedure above through step 4 then determines values for $S_{A_3A_1,A_2,\omega_k}$ and $S_{A_2A_3,A_1,\omega_k}$. We then have:

\begin{conj} \label{edgeweights} Then $S_{A_3A_1,A_2,\omega_k}=S_{A_2A_3,A_1,\omega_k}$.
\end{conj}

Provided these are equal, we then know the values for $S_{A_1A_2,A_3,\omega_k}$. This determines the arrows between the vertices on the edge $A_1A_2$ and completes the final step. The resulting cluster structure can then be glued to get a cluster structure on $\Conf_m \A$ for any integer $m$. Note that our construction was conditional upon only Conjecture~\ref{minorweight} and Conjecture~\ref{edgeweights}. Once these conjectures hold, the procedure determines the cluster structure uniquely.

Let us illustrate this in an example. For $SL_4$, the above weight analysis uniquely determines the following quiver:

\begin{center}
\begin{tikzpicture}[scale=2]

  \node (x103) at (-0.5,0.9) {\Large $\ontop{x_{103}}{\bullet}$};
  \node (x013) at (0.5,0.9) {\Large $\ontop{x_{013}}{\bullet}$};

  \node (x202) at (-1,0) {\Large $\ontop{x_{202}}{\bullet}$};
  \node (x112) at (0,0) {\Large $\ontop{x_{112}}{\bullet}$};
  \node (x022) at (1,0) {\Large $\ontop{x_{022}}{\bullet}$};

  \node (x301) at (-1.5,-0.9) {\Large $\ontop{x_{301}}{\bullet}$};
  \node (x211) at (-0.5,-0.9) {\Large $\ontop{x_{211}}{\bullet}$};
  \node (x121) at (0.5,-0.9) {\Large $\ontop{x_{121}}{\bullet}$};
  \node (x031) at (1.5,-0.9) {\Large $\ontop{x_{031}}{\bullet}$};

  \node (x310) at (-1,-1.8) {\Large $\ontop{x_{310}}{\bullet}$};
  \node (x220) at (0,-1.8) {\Large $\ontop{x_{220}}{\bullet}$};
  \node (x130) at (1,-1.8) {\Large $\ontop{x_{130}}{\bullet}$};

  \draw [->] (x013) to (x103);
  \draw [->] (x022) to (x112);
  \draw [->] (x112) to (x202);
  \draw [->] (x031) to (x121);
  \draw [->] (x121) to (x211);
  \draw [->] (x211) to (x301);
  \draw [->, dashed] (x130) to (x220);
  \draw [->, dashed] (x220) to (x310);

  \draw [->] (x130) to (x031);
  \draw [->] (x220) to (x121);
  \draw [->] (x121) to (x022);
  \draw [->] (x310) to (x211);
  \draw [->] (x211) to (x112);
  \draw [->] (x112) to (x013);
  \draw [->, dashed] (x301) to (x202);
  \draw [->, dashed] (x202) to (x103);

  \draw [->] (x301) to (x310);
  \draw [->] (x202) to (x211);
  \draw [->] (x211) to (x220);
  \draw [->] (x103) to (x112);
  \draw [->] (x112) to (x121);
  \draw [->] (x121) to (x130);
  \draw [->, dashed] (x013) to (x022);
  \draw [->, dashed] (x022) to (x031);

\draw[yshift=-2.0cm]
  node[below,text width=6cm] 
  {
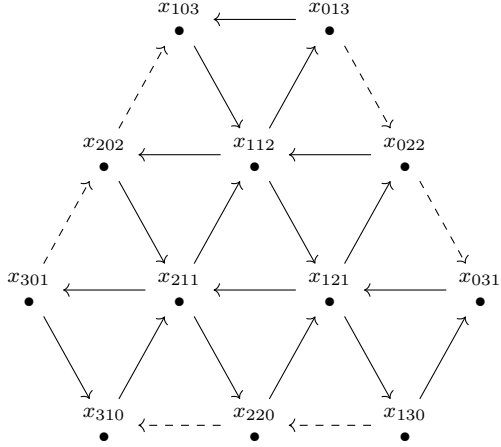
  Figure 7. The quiver for the cluster algebra on $\Conf_3 \A_{SL_4}$.
  };

\end{tikzpicture}
\end{center}

The arrows in and out of vertices $x_{211}$ and $x_{121}$ are forced by equation~\ref{face}. The arrows from $x_{301}$ to $x_{310}$ and from $x_{130}$ to $x_{031}$ are determined by equation~\ref{edge1}. One can check then that for the vertices on edges $A_1A_3$ and $A_2A_3$, equation~\ref{edge2} holds. Finally, the dotted arrows on the bottom row are determined by the fact that $\lambda_k$ and $\mu_k$ only depend on $\omega_k$.

The function at $x_{301}$ has weight $\omega_3$ at $A_1$ and weight $\omega_1$ at $A_3$. We can then calculate, for example, that
$$S_{A_3A_1,A_2,\omega_1}=(-\omega_3 + \frac{\omega_2}{2}, 0, \omega_1 - \frac{\omega_2}{2}),$$
$$S_{A_3A_1,A_2,\omega_2}=(-\omega_2 + \frac{\omega_1+\omega_3}{2}, 0, \omega_2 - \frac{\omega_1+\omega_3}{2}),$$
$$S_{A_3A_1,A_2,\omega_3}=(-\omega_1 +\frac{\omega_2}{2}, 0, \omega_3 - \frac{\omega_2}{2}).$$

The computation above suggests the following conjecture:

\begin{conj}\label{edgeweight} $S_{AB,C,\omega_k}$ has weight $\frac{\alpha_k}{2}$ at $A$ and weight $\frac{w_0(\alpha_k)}{2}$ at $B$. In other words, $\lambda_k=\frac{\alpha_k}{2}$ and $\mu_k=\frac{w_0(\alpha_k)}{2}$. 
\end{conj}

This conjecture goes beyond what we need to construct the cluster structure on $\Conf_3 \A_G$. The conjecture automatically implies that equation~\ref{edge2} holds, for it implies that $\lambda_k=\frac{\alpha_k}{2}$ and $\mu_{k^*}=\frac{w_0(\alpha_{k^*})}{2}=-\frac{\alpha_k}{2}$.

\begin{prop} The conjecture holds for classical groups.
\end{prop}

\begin{proof} This can be done by explicit calculation in type $A$ following from the calculations in the proof of Proposition~\ref{Aminors}. For other classical groups, this was verified in \cite{Le}.
\end{proof} 

Thus, we have a procedure for constructing the cluster structure on $\Conf_3 \A_G$ from any reduced word for $w_0$, along with the fractional arrows we need to glue triangles together.

\begin{rmk} Let us note that the construction of the quiver depended on Conjectures ~\ref{minorweight} and ~\ref{edgeweight} as well as the calculation of generalized minors as tensor invariants. After that the construction becomes algorithmic.

Our procedure will not work if Conjectures ~\ref{minorweight} and ~\ref{edgeweight} are false, though we are reasonably confident the conjectures hold for all simple groups.
\end{rmk}

\subsection{$S_3$ symmetries}\label{S_3}

The arguments in the preceding sections show how to derive the cluster structure on $\Conf_3 \A_G$ from the one on $G^{w_0,e}$. Note that, assuming conjectures ~\ref{minorweight} and ~\ref{edgeweight}, the former is determined by the latter.

Our next task will be to explain how the cluster structure on $\A_{G,S}$ can be constructed from that on $\Conf_3 \A_G$. We will get a cluster on $\A_{G,S}$ for any triangulation (plus some more choices). In order to relate these various clusters, we will need at the very least sequences of mutations realizing rotation symmetries and flips of triangulation.

It turns out to be convenient not only to consider rotational symmetries, but symmetries of the full group $S_3$ on $\Conf_3 \A_G$. Let us start with cyclic symmetries. Note that in order to show that the cluster structure on $\Conf_m \A_G$ is compatible with the twisted cyclic shift map (and hence that the functions most naturally live on $\A_{G,S}$), it is enough to show that show that the cluster structure is compatible with the cyclic shift map on $\Conf_3 \A_G$.

Recall that we have a cyclic shift map defined on $\Conf_3 \A_G$:
$$T: \Conf_3 \A_G \rightarrow \Conf_3 \A_G.$$
We can pull back our cluster structure along the maps $T$ and $T^2$ to get cluster structures corresponding to cyclic permutations in $S_3$.

More precisely, consider the even permutation $\sigma_{(123)} \in S_3$. $\sigma_{(123)}$ naturally acts on $\Conf_3 \A_G$. Then consider the cluster structure coming from taking the same quiver as before, except that if the function $f_i$ was originally attached to vertex $i$, we now attach the function $T^*f_i$ to the vertex $i$. This will give the cluster coming from the symmetry $\sigma_{(123)}$.

There is also a cluster structure on $\Conf_3 \A_G$ coming from any odd permutation in $S_3$. These cluster structures do not come from algebra homomorphisms (as was the case for the twisted cyclic shift map), but what are known as quasi-homomorphism (one reference for these is \cite{F}).

It turns out to be most convenient for our current purposes to deal with the transposition $\sigma_{(12)}$ (the other transpositions differ by a cyclic shift). Define the map 
$$\tau: \Conf_3 \A_G \rightarrow \Conf_3 \A_G$$
given by
$$(F_1, F_2, F_3) \rightarrow (F_2, F_1, F_3).$$
This map is definitely \emph{not} realized by an element of the cluster modular group and not a positive map. However, it does allow us to realize a cluster associated to the transposition $\sigma_{(12)}$. To obtain this cluster structure coming from this transposition, take the original quiver and reverse all the arrows. 
If the function $f_i$ was originally attached to vertex $i$, now attach the function $\pm \tau^*f_i$ to the vertex $i$. The correct sign to take here is subtle, and this subtlety is a consequence of the fact that we are dealing with a quasi-homomorphism.

Here is a more precise way to construct this cluster structure which pins down this sign. Suppose that we have a cluster constructed on the triangle $A_1A_2A_3$ using our above construction for the word 
$$w_0=s_{i_1} s_{i_2} \dots s_{i_K-1} s_{i_K}.$$
Then take the word
$$w_0=s_{i_K} s_{i_K-1} \dots s_{i_2} s_{i_1}.$$
The cluster corresponding to this new word is the same as the result of applying the reflection interchanging $A_1$ and $A_2$ to the original cluster and reversing all arrows.

The following theorem is due to Berenstein, Fomin and Zelevinsky \cite{BFZ}:

\begin{theorem}\label{reduced words} Given any two reduced word expressions for $w_0$, the resulting cluster structures on $B^-$ are related by a sequence of cluster mutations.
\end{theorem}

\begin{prop}\label{reduced} Suppose that we can construct a cluster structure on $\Conf_3 \A$ using a reduced word for $w_0$ as above. Then we can construct a cluster structure on  $\Conf_3 \A$ using any reduced word for $w_0$.
\end{prop}

\begin{proof} The reduced word for $w_0$ determines the face functions and the functions on edges $A_1A_3$ and $A_2A_3$. It also determines the (fractional) arrows except those involving a vertex on $A_1A_2$. Note that steps 1-5 in the above procedure above determined the cluster structure on $\Conf_3 \A$ uniquely. Suppose we have two reduced words $s_{i_1} s_{i_2} \dots s_{i_K-1} s_{i_K}$ and $s_{j_1} s_{j_2} \dots s_{j_K-1} s_{j_K}$, and that the reduced word $s_{i_1} s_{i_2} \dots s_{i_K-1} s_{i_K}$ gives a cluster structure on $\Conf_3 \A$. Then \cite{BFZ} gives a sequence of mutations relating the cluster structure on $B^-$ corresponding to $s_{i_1} s_{i_2} \dots s_{i_K-1} s_{i_K}$ to the cluster structure corresponding to $s_{j_1} s_{j_2} \dots s_{j_K-1} s_{j_K}$. Let us apply this same sequence of mutations to our cluster structure on $\Conf_3 \A$.

We claim that this gives the cluster structure on $\Conf_3 \A$ corresponding to $s_{j_1} s_{j_2} \dots s_{j_K-1} s_{j_K}$. The sequence of mutations only involves face vertices. Therefore throughout the sequence of mutations, the equations ~\ref{face} continue to hold, and the values of $S_{A_3A_1,A_2,\omega_k}, S_{A_2A_3,A_1,\omega_k}$ and $S_{A_1A_2,A_3,\omega_k}$ remain constant. By the uniqueness of our construction, we end up with the cluster structure corresponding to the reduced word $s_{j_1} s_{j_2} \dots s_{j_K-1} s_{j_K}$.
\end{proof}

\begin{cor} In the situations where we have a cluster algebra structure on $\Conf_3 \A$ as constructed above, the cluster coming from the transposition interchanging $A_1$ and $A_2$ can be related to the original cluster by a sequence of mutations.
\end{cor}

Thus we have constructed clusters on $\Conf_3 \A$ associated to all the $S_3$ symmetries. Let us note here, that though for any reduced word for $w_0$, we have constructed six clusters on $\Conf_3 \A$, we do not know how to relate all these clusters by sequences of mutations in general (so far we have only explained the transposition $\sigma_{(12)}$). 

In type $A$, the quiver for the reduced word $$w_0=s_1 s_2 \dots s_{n-1} s_1 s_2 \dots s_{n-2} \dots s_1 s_2 s_3 s_1 s_2 s_1$$
is magically $S_3$-symmetric, and the flip is realized by the octahedron recurrence. In types $B, C, D,$ the realization of $S_3$ symmetries via sequences of mutations was computed in \cite{Le}. We will later carry out these computations in type $G_2$.

It is perhaps instructive to verify that the quiver in type $A$ is indeed $S_3$-symmetric. The reduced word for $w_0$gives the same cluster as the reversed reduced word. This shows that the cluster is symmetric under the transposition $\sigma_{(13)}$. The cyclic symmetry is more interesting. Let us illustrate this in the case of $SL_4$; the general case is no more interesting. Let $(F_1, F_2, F_3)$ be  configuration of flags, and let $F_1$ be given by vectors $u_1, u_2, u_3$, $F_2$ be given by $v_1, v_2, v_3$, and $F_3$ be given by $w_1, w_2, w_3$. Take the function
$$u_1 \wedge v_1 \wedge w_1 \wedge w_2.$$
The map $T$ takes $(F_1, F_2, F_3)$ to $(F_2, F_3, -F_1)$. Then the pull-back of the above function is
$$v_1 \wedge w_1 \wedge -u_1 \wedge -u_2,$$
which is equal to the function
$$u_1 \wedge u_2 \wedge v_1 \wedge w_1,$$
which is another function in the original cluster.

\subsection{The cluster structure for general surfaces}

In this subsection, we show how to amalgamate the cluster structures on $\Conf_3 \A_G$ to get the cluster structure on $\A_{G,S}$ for a general surface $S$. We start with the case of $S$ a disc with $m$ marked points.

The cluster structure for $\Conf_m \A$ comes from triangulating an $m$-gon and then attaching the cluster structure on $\Conf_3 \A$ to each triangle. Let us make this more precise. On $\Conf_3 \A_{G}$, the edge vertices are frozen vertices. Attached to these vertices are functions which are invariants of tensor products $[V_{\omega_i} \otimes V_{\omega_i^{\vee}}]^G$. Because edge vertices are frozen, this is true in all the clusters for $\Conf_3 \A$.

To form the quiver for $\Conf_m \A_{G}$, we first take a triangulation of an $m$-gon. On each of the $m-2$ triangles, attach any one of the six quivers formed from performing $S_3$ symmetries on the quiver for $\Conf_3 \A_{G}$ described above. Each edge of each of these triangles has $n$ frozen vertices. Let us describe how to glue two triangles together.

Let $T_1, T_2$ be two triangles with edges $e_1, e_2, e_3$ and $e_4, e_5, e_6$, respectively. Suppose that we would like to glue the edges $e_1$ and $e_4$. $e_1$ and $e_4$ each have $n$ frozen vertices. We will glue these $2n$ vertices together in pairs to form $n$ vertices. Each frozen vertex is glued to another vertex that shares the same function. These vertices then become unfrozen. If $i$ and $i'$ are glued to give a new vertex $i''$, and if $j$ does not get glued, we declare:
$$b_{i''j}=b_{ij}+b_{i'j}.$$
Similarly, if vertices $i$ and $j$ are glued with $i'$ and $j'$ to get new vertices $i''$ and $j''$, then we declare that $$b_{i''j''}=b_{ij}+b_{ij'}+b_{i'j}+b_{i'j'}.$$
In other words, two dotted arrows in the same direction glue to give us a solid arrow, whereas two dotted arrows in the opposite direction cancel to give us no arrow. One can easily check that any gluing will result in no dotted arrows using the unfrozen vertices. The arrows involving vertices that were not previously frozen remain the same.

The cluster structure on $\A_{G,S}$ is constructed in a similar way by amalgamating the cluster structure on $\Conf_3 \A_G$. The space $\Conf_3 \A_G$ is more canonically viewed as the space of decorated local systems on a triangle. It then becomes clear that we can glue these spaces to get decorated local systems on a general surfact $S$.

In order to get see that the cluster structure is mapping class group equivariant and therefore get an embedding of the mapping class group in the cluster modular group, we need to have a sequence of mutations realizing any change in triangulation. This comes down to finding a sequence of mutations realizing a flip--a change in a quadrilateral from the triangulation using one diagonal to the triangulation using the other diagonal.

In type $A$, the flip is realized by the octahedron recurrence. In types $B, C, D,$ the realization of flips via sequences of mutations was performed in \cite{Le}. We treat $G_2$ in the next section. We have the following conjecture, that we know is true in types $A, B, C, D, G$:

\begin{conj} For any reductive group $G$, the six clusters coming from applying $S_3$ symmetries to the original cluster structure on $\Conf_3 \A$ can be realized via a sequence of mutations. Moreover, on $\Conf_4 \A$, the clusters coming from the two different triangulations can be related by a sequence of mutations. Thus we have a cluster on $\Conf_m \A$ attached to any choice of the following data: a triangulation of the $m$-gon, an ordering of the vertices in each triangle, and a choice of a reduced word for $w_0$ on each triangle. All of these clusters are related by sequences of mutations, and hence belong to the same cluster algebra.
\end{conj}

\begin{rmk} For the group $SL_2$, if one tries to perform a flip to get self-folded triangle, one finds oneself in an unexpected cluster, and one is led to consider what are called ``tagged'' triangulations \cite{FST}. For all other groups, the cluster on a self-folded triangle contains the functions one would expect. However, there is still a story that generalizes tagged triangulations, though it is a coincidence that for $SL_2$ these are realized via trying to get a self-folded triangle. In general, there are sequences of mutations realizing the $W$-action at each boundary component without marked points. These sequences of mutations were explained in \cite{GS2} for $SL_n$, and they readily generalize to other groups.
\end{rmk}

\begin{rmk} Note that it is not known in general whether the mapping class group is finite index in the cluster modular group. The cluster modular group certainly contains the $W$ action for each boundary component, as well as elements coming from outer automorphisms of $G$ (discussed in the next section). When there are marked points on the boundary, there are also braid group actions as discussed in \cite{F}. In these situations, we clearly have infinite index. Whether there are other elements in the cluster mapping class group is a mystery to the author.
\end{rmk}

\section{Dynkin automorphisms}

As an application of the construction of the previous section, let us show how to realize the action of outer automorphisms of the group $G$ on $\A_{G,S}$.

Let $\Aut(G)$ denote the group of automorphisms of the group $G$. $G$ acts on itself by conjugation, so it gives rise to an action of the group of inner automorphisms $\Inn(G)$. $\Inn(G)$ is normal inside $\Aut(G)$. Then let $\Out(G):=\Aut(G)/\Inn(G)$ denote the group of outer automorphisms of $G$. It is well known that $\Out(G)$ is a finite group given by automorphisms of the Dynkin diagram of $G$.

We can realize the action of $\Out(G)$ on $G$ in the following way: $\Out(G)$ acts on the Dynkin diagram, so it permutes the simple roots $\alpha$. This induces a natural action of $\Out(G)$ on the Cartan subgroup $H$. It also induces an action of $\Out(G)$ on the generators $x_{\alpha}, y_{\alpha}$. This gives an action of $\Out(G)$ on $G$, which is generated by $H, x_{\alpha}, y_{\alpha}$

The above discussion means that $\Out(G)$ naturally acts on the flag varieties $G/U$ and $G/B$. It also acts on local systems in the following way. Let $\rho: \pi_1(S) \rightarrow G$ be a representation. An element $\sigma \in \Out(G)$ acts on the set of such representations by post-composition:
$$\rho: \pi_1(S) \rightarrow G \xrightarrow{\sigma} G.$$
These actions of $\Out(G)$ on flag varieties and local systems are compatible, and hence give rise to actions on $\A_{G,S}$ and $\X_{G,S}$. Let $\sigma \in \Out(G)$. Then pulling back functions via $\sigma$ gives an action of $\sigma$ on $\mathcal{O}(\A_{G,S})$. We are interested in whether the action of $\sigma$ preserves the cluster structure. In other words, acting by $\sigma$ gives a potentially different cluster structure on $\A_{G,S}$. We would like to show that we in fact get the same cluster structure.

In order to show that $\sigma$ preserves the cluster structure on $\A_{G,S}$, it is enough to show the following: there is some seed $(I,I_0,B,d)$ with functions $f_i \in \A_{G,S}$ attached to each $i \in I$ and a sequence of mutations taking the seed $(I,I_0,B,d)$ to an isomorphic seed $(I',I'_0,B',d')$ such that the functions attached to these two seeds are related by $\sigma$. More precisely, let $\phi: (I,I_0,B,d) \rightarrow (I',I'_0,B',d')$ be such an isomorphism. For any $i \in I$, let the associated function be $f_i$. Then we would like to show that
$$\sigma^*{f_i}=f_{\phi(i)}.$$

\begin{theorem} Let $\sigma \in \Out(G)$. Then $\sigma$ acts on $\A_{G,S}$. Suppose that $\mathcal{O}(\A_{G,S})$ has the structure of a cluster algebra as constructed by the procedure given in Section 3. Then there exists a seed consisting of the functions $$f_1, f_2, \dots, f_n$$ and a sequence of mutations which transforms this seed into the seed consisting of the functions $$\sigma^*f_1, \sigma^*f_2, \dots, \sigma^*f_n$$ in such a way that the other corresponding seed data are isomorphic.
\end{theorem}

\begin{rmk} Once one has the above theorem for one seed, the theorem for any seed is automatic by mutation. The conditions of the theorem hold in types $A$ and $D_4$ \cite{FG1}, \cite{Le}. They are expected to hold in type $E_6$.
\end{rmk}

\begin{proof} It is enough to realize $\sigma$ via a sequence of mutations on each triangle.

The automorphism $\sigma$ permutes the vertices of the Dynkin diagram, so that it permutes the fundamental weights. Thus we get a natural action of $\sigma$ on the simple reflections $s_i \in W$. Consider the seed constructed from a reduced word for $w_0$:
$$w_0=s_{i_1} s_{i_2} \dots s_{i_K-1} s_{i_K}.$$
Because $\sigma$ is an automorphism of the Dynkin diagram, there is another reduced word
$$w_0=s_{\sigma(i_1)} s_{\sigma(i_2)} \dots s_{\sigma(i_{K-1})} s_{\sigma(i_K)}.$$
This reduced word gives a second seed.

By \cite{BFZ} and \label{reduced}, there is a sequence of mutations relating the seeds coming from these two reduced words. We then clearly have
$$\sigma^*(\Delta_{\sigma(u_{l})\omega_{\sigma(i_l)}, \omega_{\sigma(i_l)}})=\Delta_{u_{l}\omega_{i_l}, \omega_{i_l}}$$
and a similar formula holds for all the cluster functions in our two seeds. We then only need to verify that the seed data are isomorphic. But the seed data were determined by steps 1-5 in our procedure, which was completely Lie-theoretic, so that the seed data for the two seeds must be isomorphic.
\end{proof}

The above theorem applies to the cases when $G$ has type $A$ or $D_4$. We also expect it to apply for the automophism of $E_6$, dependent upon conjectures ~\ref{minorweight} and ~\ref{edgeweights}. In the next section, we show that there is actually a seed for the cluster structure on $\Conf_3 \A_{Spin_8,S}$ which is preserved by the outer automorphisms of $Spin_8$.

One interesting interpretation of the above theorem is that the outer automorphism $\sigma \in \Out(G)$ gives rise to an element of the \emph{cluster modular group}. This group has been well-studied, for example in \cite{FG2}, \cite{GS2}, \cite{F}. Outside of a small number of cases, the cluster modular group of $\A_{G,S}$ is known to include the mapping class group of the surface $S$. Work of Goncharov and Shen \cite{GS2} as well as their forthcoming work shows that it also contains a copy of the Weyl group $W$ for every hole of $S$ without marked points. Together, the mapping class group, the copies of $W$, and the outer automorphisms $\sigma \in \Out(G)$ give all known elements of the cluster modular group.

\begin{rmk} Let us say something the case when $G$ has type $A$. In \cite{Hen}, there is a sequence of mutations on $\Conf_3 \A_{SL_N}$ realizing the outer automorphism of $SL_N$. We will call this sequence of mutations the \emph{cactus sequence}.
We can interpret this sequence of mutations as relating the cluster structure coming from the reduced word 
$$s_1 s_2 \dots s_{N-1} s_1 s_2 \dots s_{N-2} \dots s_1 s_2 s_3 s_1 s_2 s_1$$
to the cluster structure coming from the reduced word 
$$s_{N-1} s_{N-2} \dots s_{1} s_{N-1} s_{N-2} \dots s_{2} \dots s_{N-1} s_{N-2} s_{N-3} s_{N-1} s_{N-2} s_{N-1}.$$
\end{rmk}

\section{$G$ has type $G_2$}

Throughout this section, $G$ will be the simply-connected group with type $G_2$. The root system for $G$ contains $12$ roots, and there are two simple roots, which we will call $\alpha$ and $\beta$. We normalize so that the short root $\alpha$ has length $1$, while the long root $\beta$ has length $3$. Let the corresponding simple reflections in the Weyl group be $s_a$ and $s_b$.

We will consider the quiver for $\Conf_3 \A_{G}$ coming from the reduced word
$$w_0=s_bs_as_bs_as_bs_a.$$
This yields the following quiver for $B^-$:

\begin{center}
\begin{tikzpicture}[scale=2]

  \node (xa0) at (-1.5,0) {\Large $\ontop{x_{a0}}{\bullet}$};
  \node (xa1) at (-0.5,0) {\Large $\ontop{x_{a1}}{\bullet}$};
  \node (xa2) at (0.5,0) {\Large $\ontop{x_{a2}}{\bullet}$};
  \node (xa3) at (1.5,0) {\Large $\ontop{x_{a3}}{\bullet}$};

  \node (xb0) at (-1.5,-1) {\Large $\ontop{x_{b0}}{\circ}$};
  \node (xb1) at (-0.5,-1) {\Large $\ontop{x_{b1}}{\circ}$};
  \node (xb2) at (0.5,-1) {\Large $\ontop{x_{b2}}{\circ}$};
  \node (xb3) at (1.5,-1) {\Large $\ontop{x_{b3}}{\circ}$};

  \draw [->] (xa1) to (xa0);
  \draw [->] (xa2) to (xa1);
  \draw [->] (xa3) to (xa2);

  \draw [->] (xb1) to (xb0);
  \draw [->] (xb2) to (xb1);
  \draw [->] (xb3) to (xb2);

  \draw [->] (xb0) to (xa1);
  \draw [->] (xb1) to (xa2);
  \draw [->] (xb2) to (xa3);

  \draw [->, dashed] (xa0) to (xb0);
  \draw [->] (xa1) to (xb1);
  \draw [->] (xa2) to (xb2);
  \draw [->, dashed] (xa3) to (xb3);

\draw[yshift=-1.5cm]
  node[below,text width=6cm] 
  {
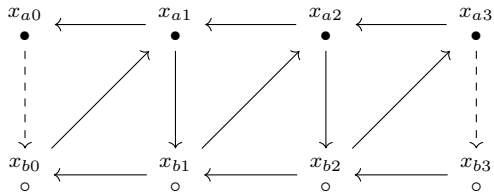
  Figure 8. The quiver for the cluster algebra on $B^-$ for $G$ of type $G_2$.
  };

\end{tikzpicture}
\end{center}

Here is how it corresponds to the reduced word for $w_0$:

\begin{center}
\begin{tikzpicture}[scale=2]

  \node (xa0) at (-1.5,0) {${\bullet}$};
  \node (xa1) at (-0.5,0) {${\bullet}$};
  \node (xa2) at (0.5,0) {${\bullet}$};
  \node (xa3) at (1.5,0) {${\bullet}$};

  \node (xb0) at (-1.5,-1) {${\circ}$};
  \node (xb1) at (-0.5,-1) {${\circ}$};
  \node (xb2) at (0.5,-1) {${\circ}$};
  \node (xb3) at (1.5,-1) {${\circ}$};

  \draw [->, postaction={decorate}] (xa1) -- (xa0) node [midway, above] {$s_a$};
  \draw [->, postaction={decorate}] (xa2) -- (xa1) node [midway, above] {$s_a$};
  \draw [->, postaction={decorate}] (xa3) -- (xa2) node [midway, above] {$s_a$};

  \draw [->, postaction={decorate}] (xb1) -- (xb0) node [midway, above] {$s_b$};
  \draw [->, postaction={decorate}] (xb2) -- (xb1) node [midway, above] {$s_b$};
  \draw [->, postaction={decorate}] (xb3) -- (xb2) node [midway, above] {$s_b$};

  \draw [->] (xb0) to (xa1);
  \draw [->] (xb1) to (xa2);
  \draw [->] (xb2) to (xa3);

  \draw [->, dashed] (xa0) to (xb0);
  \draw [->] (xa1) to (xb1);
  \draw [->] (xa2) to (xb2);
  \draw [->, dashed] (xa3) to (xb3);

\draw[yshift=-1.5cm]
  node[below,text width=6cm] 
  {
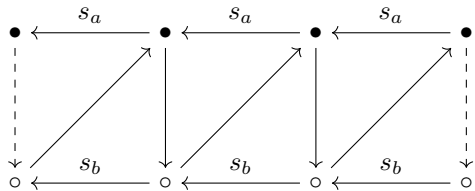
  Figure 9. How the quiver for $B^-_{G}$ corresponds to the reduced word $s_bs_as_bs_as_bs_a$.

  };

\end{tikzpicture}
\end{center}

\subsection{The functions} We now need to calculate the functions attached to the each vertex in the quiver. We will do this by relating the functions on $B^-_{G}$ (and also $\Conf_3 \A_{G}$) to those on $B^-_{Spin_8}$ (and $\Conf_3 \A_{Spin_8}$).

Note that $Spin_8$ has a Dynkin diagram with four nodes, with the central node connected to the three others. The group $S_3$ therefore acts by automorphisms on the Dynkin diagram, so that $Out(Spin_8) \simeq S_3$. The group $G$ embeds into $Spin_8$ as the fixed point set of the outer automorphism group. The embedding
$$G \xhookrightarrow{i} Spin_8$$
gives an embedding 
$$\Conf_3 \A_{G} \xhookrightarrow{j} \Conf_3 \A_{Spin_8}.$$
Then all functions in the cluster algebra on $\Conf_3 \A_G$ are pulled back from functions in the cluster algebra on $\Conf_3 \A_{Spin_8}.$ Let us elaborate upon this.

The reduced word $s_bs_as_bs_as_bs_a$ comes from a folding of a reduced word for the longest element of the Weyl group of $Spin_8$. The group $Spin_8$ has four simple roots. Let us call them $\alpha_1, \alpha_2, \alpha_3, \beta$, where the $\alpha_i$ fold to give the short root $\alpha$. Let the corresponding simple reflections be $s_{a_1}, s_{a_2}, s_{a_3}, s_{b}$. Then a reduced word for $w_0$ for $Spin_8$ is given by
$$w_0=s_bs_{a_1}s_{a_2}s_{a_3}s_bs_{a_1}s_{a_2}s_{a_3}s_bs_{a_1}s_{a_2}s_{a_3}.$$
Then, using the constructions of the previous section, this reduced word gives a cluster for $\Conf_3 \A_{Spin_8}.$ The quiver for $B^-_{Spin_8}$ is depicted below.

\begin{center}
\begin{tikzpicture}[scale=2]

  \node (xa10) at (-1.5,0) {\Large $\ontop{x_{a_10}}{\bullet}$};
  \node (xa11) at (-0.5,0) {\Large $\ontop{x_{a_11}}{\bullet}$};
  \node (xa12) at (0.5,0) {\Large $\ontop{x_{a_12}}{\bullet}$};
  \node (xa13) at (1.5,0) {\Large $\ontop{x_{a_13}}{\bullet}$};

  \node (xa20) at (-2.3,-2) {\Large $\ontop{x_{a_20}}{\bullet}$};
  \node (xa21) at (-1.3,-2) {\Large $\ontop{x_{a_21}}{\bullet}$};
  \node (xa22) at (-0.3,-2) {\Large $\ontop{x_{a_22}}{\bullet}$};
  \node (xa23) at (0.7,-2) {\Large $\ontop{x_{a_23}}{\bullet}$};

  \node (xa30) at (-0.5,-2.3) {\Large $\ontop{x_{a_30}}{\bullet}$};
  \node (xa31) at (0.5,-2.3) {\Large $\ontop{x_{a_31}}{\bullet}$};
  \node (xa32) at (1.5,-2.3) {\Large $\ontop{x_{a_32}}{\bullet}$};
  \node (xa33) at (2.5,-2.3) {\Large $\ontop{x_{a_33}}{\bullet}$};

  \node (xb0) at (-1.5,-1) {\Large $\ontop{x_{b0}}{\bullet}$};
  \node (xb1) at (-0.5,-1) {\Large $\ontop{x_{b1}}{\bullet}$};
  \node (xb2) at (0.5,-1) {\Large $\ontop{x_{b2}}{\bullet}$};
  \node (xb3) at (1.5,-1) {\Large $\ontop{x_{b3}}{\bullet}$};

  \draw [->, postaction={decorate}] (xb1) -- (xb0) node [midway, above] {};
  \draw [->, postaction={decorate}] (xb2) -- (xb1) node [midway, above] {};
  \draw [->, postaction={decorate}] (xb3) -- (xb2) node [midway, above] {};

  \draw [->, postaction={decorate}] (xa11) -- (xa10) node [midway, above] {};
  \draw [->, postaction={decorate}] (xa12) -- (xa11) node [midway, above] {};
  \draw [->, postaction={decorate}] (xa13) -- (xa12) node [midway, above] {};

  \draw [->] (xb0) to (xa11);
  \draw [->] (xb1) to (xa12);
  \draw [->] (xb2) to (xa13);

  \draw [->, dashed] (xa10) to (xb0);
  \draw [->] (xa11) to (xb1);
  \draw [->] (xa12) to (xb2);
  \draw [->, dashed] (xa13) to (xb3);

  \draw [->, postaction={decorate}] (xa21) -- (xa20) node [midway, above] {};
  \draw [->, postaction={decorate}] (xa22) -- (xa21) node [midway, above] {};
  \draw [->, postaction={decorate}] (xa23) -- (xa22) node [midway, above] {};

  \draw [->] (xb0) to (xa21);
  \draw [->] (xb1) to (xa22);
  \draw [->] (xb2) to (xa23);

  \draw [->, dashed] (xa20) to (xb0);
  \draw [->] (xa21) to (xb1);
  \draw [->] (xa22) to (xb2);
  \draw [->, dashed] (xa23) to (xb3);

  \draw [->, postaction={decorate}] (xa31) -- (xa30) node [midway, above] {};
  \draw [->, postaction={decorate}] (xa32) -- (xa31) node [midway, above] {};
  \draw [->, postaction={decorate}] (xa33) -- (xa32) node [midway, above] {};

  \draw [->] (xb0) to (xa31);
  \draw [->] (xb1) to (xa32);
  \draw [->] (xb2) to (xa33);

  \draw [->, dashed] (xa30) to (xb0);
  \draw [->] (xa31) to (xb1);
  \draw [->] (xa32) to (xb2);
  \draw [->, dashed] (xa33) to (xb3);

\draw[yshift=-2.7cm]
  node[below,text width=6cm] 
  {
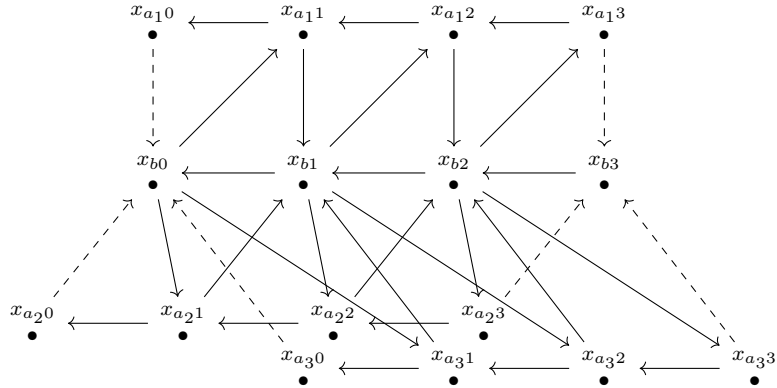
  Figure 10. The cluster algebra for $B^-_{Spin_8}$.
  };

\end{tikzpicture}
\end{center}

This particular cluster is invariant under the action of $Out(Spin_8)$. It should be clear how $S_3$ acts on this quiver: There are vertices $x_{a_ij}$ for $1 \leq i \leq 3$ and $0 \leq j \leq 3$. The vertices $x_{a_ij}$ for fixed $j$ are permuted among each other. The action of $\sigma \in S_3$ maps  $x_{a_ij} \rightarrow x_{a_{\sigma(i)}j}$.

\begin{center}
\begin{tikzpicture}[scale=2]

  \node (xa10) at (-1.5,0) {${\bullet}$};
  \node (xa11) at (-0.5,0) {${\bullet}$};
  \node (xa12) at (0.5,0) {${\bullet}$};
  \node (xa13) at (1.5,0) {${\bullet}$};

  \node (xa20) at (-2.3,-2) {${\bullet}$};
  \node (xa21) at (-1.3,-2) {${\bullet}$};
  \node (xa22) at (-0.3,-2) {${\bullet}$};
  \node (xa23) at (0.7,-2) {${\bullet}$};

  \node (xa30) at (-0.5,-2.3) {${\bullet}$};
  \node (xa31) at (0.5,-2.3) {${\bullet}$};
  \node (xa32) at (1.5,-2.3) {${\bullet}$};
  \node (xa33) at (2.5,-2.3) {${\bullet}$};

  \node (xb0) at (-1.5,-1) {${\bullet}$};
  \node (xb1) at (-0.5,-1) {${\bullet}$};
  \node (xb2) at (0.5,-1) {${\bullet}$};
  \node (xb3) at (1.5,-1) {${\bullet}$};

  \draw [->, postaction={decorate}] (xb1) -- (xb0) node [midway, above] {$s_b$};
  \draw [->, postaction={decorate}] (xb2) -- (xb1) node [midway, above] {$s_b$};
  \draw [->, postaction={decorate}] (xb3) -- (xb2) node [midway, above] {$s_b$};

  \draw [->, postaction={decorate}] (xa11) -- (xa10) node [midway, above] {$s_{a_1}$};
  \draw [->, postaction={decorate}] (xa12) -- (xa11) node [midway, above] {$s_{a_1}$};
  \draw [->, postaction={decorate}] (xa13) -- (xa12) node [midway, above] {$s_{a_1}$};

  \draw [->] (xb0) to (xa11);
  \draw [->] (xb1) to (xa12);
  \draw [->] (xb2) to (xa13);

  \draw [->, dashed] (xa10) to (xb0);
  \draw [->] (xa11) to (xb1);
  \draw [->] (xa12) to (xb2);
  \draw [->, dashed] (xa13) to (xb3);

  \draw [->, postaction={decorate}] (xa21) -- (xa20) node [midway, above] {$s_{a_2}$};
  \draw [->, postaction={decorate}] (xa22) -- (xa21) node [midway, above] {$s_{a_2}$};
  \draw [->, postaction={decorate}] (xa23) -- (xa22) node [midway, above] {$s_{a_2}$};

  \draw [->] (xb0) to (xa21);
  \draw [->] (xb1) to (xa22);
  \draw [->] (xb2) to (xa23);

  \draw [->, dashed] (xa20) to (xb0);
  \draw [->] (xa21) to (xb1);
  \draw [->] (xa22) to (xb2);
  \draw [->, dashed] (xa23) to (xb3);

  \draw [->, postaction={decorate}] (xa31) -- (xa30) node [midway, above] {$s_{a_3}$};
  \draw [->, postaction={decorate}] (xa32) -- (xa31) node [midway, above] {$s_{a_3}$};
  \draw [->, postaction={decorate}] (xa33) -- (xa32) node [midway, above] {$s_{a_3}$};

  \draw [->] (xb0) to (xa31);
  \draw [->] (xb1) to (xa32);
  \draw [->] (xb2) to (xa33);

  \draw [->, dashed] (xa30) to (xb0);
  \draw [->] (xa31) to (xb1);
  \draw [->] (xa32) to (xb2);
  \draw [->, dashed] (xa33) to (xb3);

\draw[yshift=-2.7cm]
  node[below,text width=6cm] 
  {
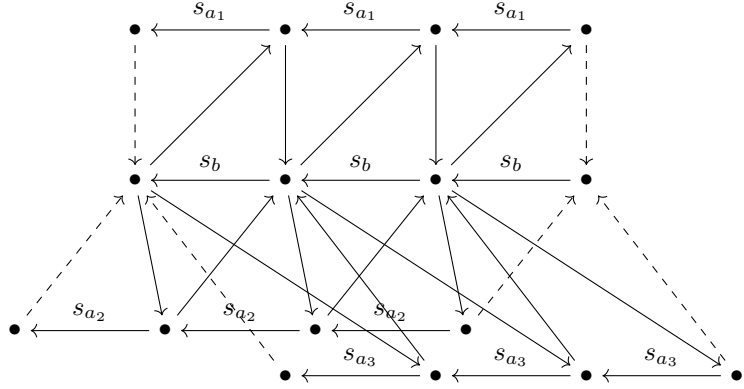
  Figure 11.  How the quiver  for $B^-_{Spin_8}$ corresponds to the reduced word $w_0=s_bs_{a_1}s_{a_2}s_{a_3}s_bs_{a_1}s_{a_2}s_{a_3}s_bs_{a_1}s_{a_2}s_{a_3}$.
  };

\end{tikzpicture}
\end{center}

The functions attached to the above cluster can be calculated by reduced minors. Alternatively, the above cluster is related to the one constructed in \cite{Le} by a sequence of mutations--since both clusters come from a reduced word we can use Theorem~\ref{reduced words}. Therefore the functions attached to the cluster above can be written directly in terms of the functions from \cite{Le}. Let us now explain how this allows us to construct the functions on $B^-_{G}$ and $\Conf_3 \A_{G}$ from those on $\Conf_3 \A_{Spin_8}$.

Suppose that the cluster function attached to a vertex $x_i$ is $f_i$. The action of $\sigma \in Out(G)$ then acts on functions by $\sigma^*(f_{a_ij}) = f_{a_{\sigma(i)}j}$. Then $B^-_{G} \subset B^-_{Spin_8}$ is precisely the locus of points where $f_{a_1j}=f_{a_2j}=f_{a_3j}$ for all $j$. Moreover, we can pull back the function $f_{a_ij}$ to $B^-_{G}$ to get the function $f_{aj}$.

Using the constructions of the previous section, we can complete the cluster algebra structure on $B^-_{Spin_8}$ to obtain one on $\Conf_3 \A_{Spin_8}$. By slight abuse of notation, we will still denote by $f_{a_ij}$ the function $f_{a_ij}$ extended to $\Conf_3 \A_{Spin_8}$. Then the functions on $\Conf_3 \A_{Spin_8}$ will be $f_{a_ij}$ for $1 \leq i \leq 3$ and $0 \leq j \leq 3$ plus the functions on the third edge, which are analogous to the functions on the other two edges. Because our construction was Lie-theoretic, we still have an action of $Out(Spin_8)$ this cluster for $\Conf_3 \A_{Spin_8}$. We can fold this cluster by the automorphisms to get the cluster with the quiver in Figure 12. In particular, the quiver contains information on how the edge vertices along the third edge attach to the cluster for $B^-$.

\begin{center}
\begin{tikzpicture}[scale=2]

  \node (xa0) at (-1.5,0) {\Large $\ontop{x_{a0}}{\bullet}$};
  \node (xa1) at (-0.5,0) {\Large $\ontop{x_{a1}}{\bullet}$};
  \node (xa2) at (0.5,0) {\Large $\ontop{x_{a2}}{\bullet}$};
  \node (xa3) at (1.5,0) {\Large $\ontop{x_{a3}}{\bullet}$};

  \node (xb0) at (-1.5,-1) {\Large $\ontop{x_{b0}}{\circ}$};
  \node (xb1) at (-0.5,-1) {\Large $\ontop{x_{b1}}{\circ}$};
  \node (xb2) at (0.5,-1) {\Large $\ontop{x_{b2}}{\circ}$};
  \node (xb3) at (1.5,-1) {\Large $\ontop{x_{b3}}{\circ}$};

  \node (xa) at (-0.5,1) {\Large $\ontop{x_{a}}{\bullet}$};
  \node (xb) at (0.5,-2) {\Large $\ontop{x_{b}}{\circ}$};

  \draw [->] (xa1) to (xa0);
  \draw [->] (xa2) to (xa1);
  \draw [->] (xa3) to (xa2);

  \draw [->] (xb1) to (xb0);
  \draw [->] (xb2) to (xb1);
  \draw [->] (xb3) to (xb2);

  \draw [->] (xb0) to (xa1);
  \draw [->] (xb1) to (xa2);
  \draw [->] (xb2) to (xa3);
  \draw [->] (xa0) to (xa);
  \draw [->] (xb) to (xb3);

  \draw [->] (xb2) to (xb);
  \draw [->] (xa) to (xa1);

  \draw [->, dashed] (xb) .. controls +(0:2) and +(0:4) ..  (xa);

  \draw [->, dashed] (xa0) to (xb0);
  \draw [->] (xa1) to (xb1);
  \draw [->] (xa2) to (xb2);
  \draw [->, dashed] (xa3) to (xb3);

\draw[yshift=-2.5cm]
  node[below,text width=6cm] 
  {
  Figure 12. The quiver for the cluster algebra on $\Conf_3 \A_G$ for $G$ of type $G_2$.
  };

\end{tikzpicture}
\end{center}

If the function $f_i$ is attached to the vertex $x_i$, we will have that 
$$j^*(f_{a_ij})=f_{aj}.$$

In figure 13, we put functions at the corresponding vertices. We will abbreviate by $$\tcfr{\nu}{\lambda}{\mu}$$
a function which lies in the invariant space 
$$[V_{\lambda} \otimes V_{\mu} \otimes V_{\nu}]^G,$$
where $\lambda, \mu, \nu$ are the gradings of the function in the space of functions on the flags $A_1, A_2, A_3$ respectively. We picture the three flags lying at the vertices of a triangle as follows:

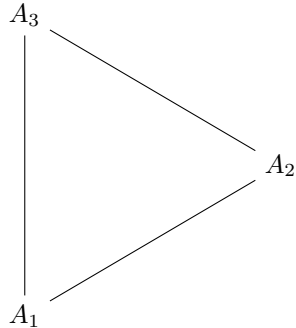
\begin{center}
\begin{tikzpicture}[scale=2]

  \node (A_1) at (-0.7,-1) {$A_1$};
  \node (A_2) at (1,0) {$A_2$};
  \node (A_3) at (-0.7,1) {$A_3$};

  \draw [] (A_1) -- (A_2);
  \draw [] (A_2) -- (A_3);
  \draw [] (A_3) -- (A_1);

\end{tikzpicture}
\end{center}

For the remainder of this paper, we will not specify which particular invariant vector within $[V_{\lambda} \otimes V_{\mu} \otimes V_{\nu}]^G$ the function corresponds to. The particular functions can be calculated in terms of the functions in \cite{Le}. We emphasize the weights of these functions because they are what are important for determining the quiver using the algorithm prescribed in Section 4, as well as for describing the $\X$-space. We will denote by $a$ the fundamental weight $\omega_{\alpha}$ and by $b$ the fundamental weight $\omega_{\beta}$.

\begin{center}
\begin{tikzpicture}[scale=2]

  \node (xa0) at (-1.5,0) {$\tcfr{a}{a}{}$};
  \node (xa1) at (-0.5,0) {$\tcfr{a}{b}{a}$};
  \node (xa2) at (0.5,0) {$\tcfr{a}{b}{2a}$};
  \node (xa3) at (1.5,0) {$\tcfr{a}{}{a}$};

  \node (xb0) at (-1.5,-1) {$\tcfr{b}{b}{}$};
  \node (xb1) at (-0.5,-1) {$\tcfr{b}{2b}{3a}$};
  \node (xb2) at (0.5,-1) {$\tcfr{b}{b}{3a}$};
  \node (xb3) at (1.5,-1) {$\tcfr{b}{}{b}$};

  \node (xa) at (-0.5,1) {$\tcfr{}{a}{a}$};
  \node (xb) at (0.5,-2) {$\tcfr{}{b}{b}$};

  \draw [->] (xa1) to (xa0);
  \draw [->] (xa2) to (xa1);
  \draw [->] (xa3) to (xa2);

  \draw [->] (xb1) to (xb0);
  \draw [->] (xb2) to (xb1);
  \draw [->] (xb3) to (xb2);

  \draw [->] (xb0) to (xa1);
  \draw [->] (xb1) to (xa2);
  \draw [->] (xb2) to (xa3);
  \draw [->] (xa0) to (xa);
  \draw [->] (xb) to (xb3);

  \draw [->] (xb2) to (xb);
  \draw [->] (xa) to (xa1);

  \draw [->, dashed] (xb) .. controls +(0:2) and +(0:4) ..  (xa);

  \draw [->, dashed] (xa0) to (xb0);
  \draw [->] (xa1) to (xb1);
  \draw [->] (xa2) to (xb2);
  \draw [->, dashed] (xa3) to (xb3);

\draw[yshift=-2.5cm]
  node[below,text width=6cm] 
  {
  Figure 13. The functions for the cluster algebra on $\Conf_3 \A_G$ for $G$ of type $G_2$.
  };

\end{tikzpicture}
\end{center}

\begin{rmk} Note that in the construction of the cluster structure on $\Conf_3 \A_G$ for $G=G_2$, we were able to work around calculating with generalized minors by using the embedding
$$\Conf_3 \A_{G} \hookrightarrow \Conf_3 \A_{Spin_8}$$
and pulling back functions. Cluster folding allows us to relate sequences of mutations on $\Conf_3 \A_G$ for $G=G_2$ with those on $\Conf_3 \A_{Spin_8}$. This allows us to both identify the functions and compute which tensor invariant spaces they lie in. As far as we know, a similar work-around is not available in types $E$ and $F$, and one must carry out the computation more directly. 
\end{rmk}

\subsection{The first transposition}

We now exhibit the sequence of mutations associated to the transposition $(13) \in S_3$. This transposition is realized by the mutation sequence 

\begin{equation} \label{13}
\begin{gathered}
x_{a2}, \\
x_{a1}, x_{b1}, \\
x_{a2} \\
\end{gathered}
\end{equation}

We view the sequence of mutations as consisting of three stages, corresponding to the three rows in which we listed the mutations. The quiver and the functions transform as follows:

\begin{center}
\begin{tikzpicture}[scale=2]

  \node (xa0) at (-1.5,0) {$\tcfr{a}{a}{}$};
  \node (xa1) at (-0.5,0) {$\tcfr{a}{b}{a}$};
  \node (xa2) at (0.5,0) {$\tcfr{b}{b}{2a}$};
  \node (xa3) at (1.5,0) {$\tcfr{a}{}{a}$};

  \node (xb0) at (-1.5,-1) {$\tcfr{b}{b}{}$};
  \node (xb1) at (-0.5,-1) {$\tcfr{b}{2b}{3a}$};
  \node (xb2) at (0.5,-1) {$\tcfr{b}{b}{3a}$};
  \node (xb3) at (1.5,-1) {$\tcfr{b}{}{b}$};

  \node (xa) at (-0.5,1) {$\tcfr{}{a}{a}$};
  \node (xb) at (0.5,-2) {$\tcfr{}{b}{b}$};

  \draw [->] (xa1) to (xa0);
  \draw [->] (xa1) to (xa2);
  \draw [->] (xa2) to (xa3);
  \draw [->] (xa3) .. controls +(135:1) and +(45:1) ..  (xa1);

  \draw [->] (xb1) to (xb0);
  \draw [->] ([yshift=1.5pt, xshift=-1pt] xb1.east) to ([yshift=1.5pt, xshift=1pt] xb2.west);
  \draw [->] ([yshift=-1.5pt, xshift=-1pt] xb1.east) to ([yshift=-1.5pt, xshift=1pt] xb2.west);
  \draw [->] (xb3) to (xb2);

  \draw [->] (xb0) to (xa1);
  \draw [->] (xa2) to (xb1);
  \draw [->] (xa0) to (xa);
  \draw [->] (xb) to (xb3);

  \draw [->] (xb2) to (xb);
  \draw [->] (xa) to (xa1);

  \draw [->, dashed] (xb) .. controls +(0:2) and +(0:4) ..  (xa);

  \draw [->, dashed] (xa0) to (xb0);
  \draw [->] (xb2) to (xa2);
  \draw [->, dashed] (xa3) to (xb3);

\draw[yshift=-2.5cm]
  node[below,text width=6cm] 
  {
  Figure 14a. The functions for the cluster algebra on $\Conf_3 \A_G$ after performing the first mutation.
  };

\end{tikzpicture}
\end{center}

\begin{center}
\begin{tikzpicture}[scale=2]

  \node (xa0) at (-1.5,0) {$\tcfr{a}{a}{}$};
  \node (xa1) at (-0.5,0) {$\tcfr{b}{a}{a}$};
  \node (xa2) at (0.5,0) {$\tcfr{b}{b}{2a}$};
  \node (xa3) at (-0.5,1) {$\tcfr{a}{}{a}$};

  \node (xb0) at (-1.5,-1) {$\tcfr{b}{b}{}$};
  \node (xb1) at (-0.5,-1) {$\tcfr{2b}{b}{3a}$};
  \node (xb2) at (0.5,-1) {$\tcfr{b}{b}{3a}$};
  \node (xb3) at (0.5,-2) {$\tcfr{b}{}{b}$};

  \node (xa) at (1.5,0) {$\tcfr{}{a}{a}$};
  \node (xb) at (1.5,-1) {$\tcfr{}{b}{b}$};

  \draw [->] (xa0) to (xa1);
  \draw [->] (xa2) to (xa1);
  \draw [->] (xa) to  (xa2);

  \draw [->] (xa1) to  (xa3);
  \draw [->] (xa3) to  (xa0);

  \draw [->] (xb0) to (xb1);
  \draw [->] ([yshift=1.5pt, xshift=1pt] xb2.west) to ([yshift=1.5pt, xshift=-1pt] xb1.east);
  \draw [->] ([yshift=-1.5pt, xshift=1pt] xb2.west) to ([yshift=-1.5pt, xshift=-1pt] xb1.east);
  \draw [->] (xb3) to (xb2);

  \draw [->] (xa1) to (xb0);
  \draw [->] (xb1) to (xa2);
  \draw [->] (xb) to (xb3);

  \draw [->] (xb2) to (xb);
  \draw [->] (xa1) .. controls +(45:1) and +(135:1) .. (xa);

  \draw [->, dashed] (xb) to (xa);

  \draw [->, dashed] (xb0) to (xa0);
  \draw [->] (xa2) to (xb2);
  \draw [->, dashed] (xa3) .. controls +(0:4) and +(0:2) ..   (xb3);

\draw[yshift=-2.5cm]
  node[below,text width=6cm] 
  {
  Figure 14b. The functions for the cluster algebra on $\Conf_3 \A_G$ after performing the third mutation.
  };

\end{tikzpicture}
\end{center}

\begin{center}
\begin{tikzpicture}[scale=2]

  \node (xa0) at (-1.5,0) {$\tcfr{a}{a}{}$};
  \node (xa1) at (-0.5,0) {$\tcfr{b}{a}{a}$};
  \node (xa2) at (0.5,0) {$\tcfr{b}{a}{2a}$};
  \node (xa3) at (-0.5,1) {$\tcfr{a}{}{a}$};

  \node (xb0) at (-1.5,-1) {$\tcfr{b}{b}{}$};
  \node (xb1) at (-0.5,-1) {$\tcfr{2b}{b}{3a}$};
  \node (xb2) at (0.5,-1) {$\tcfr{b}{b}{3a}$};
  \node (xb3) at (0.5,-2) {$\tcfr{b}{}{b}$};

  \node (xa) at (1.5,0) {$\tcfr{}{a}{a}$};
  \node (xb) at (1.5,-1) {$\tcfr{}{b}{b}$};

  \draw [->] (xa0) to (xa1);
  \draw [->] (xa1) to (xa2);
  \draw [->] (xa2) to  (xa);

  \draw [->] (xa1) to  (xa3);
  \draw [->] (xa3) to  (xa0);

  \draw [->] (xb0) to (xb1);
  \draw [->] (xb1) to (xb2);
  \draw [->] (xb3) to (xb2);

  \draw [->] (xa1) to (xb0);
  \draw [->] (xa2) to (xb1);
  \draw [->] (xa) to (xb2);
  \draw [->] (xb) to (xb3);

  \draw [->] (xb2) to (xb);

  \draw [->, dashed] (xb) to (xa);

  \draw [->, dashed] (xb0) to (xa0);
  \draw [->] (xb1) to (xa1);
  \draw [->] (xb2) to (xa2);
  \draw [->, dashed] (xa3) .. controls +(0:4) and +(0:2) ..   (xb3);

\draw[yshift=-2.5cm]
  node[below,text width=6cm] 
  {
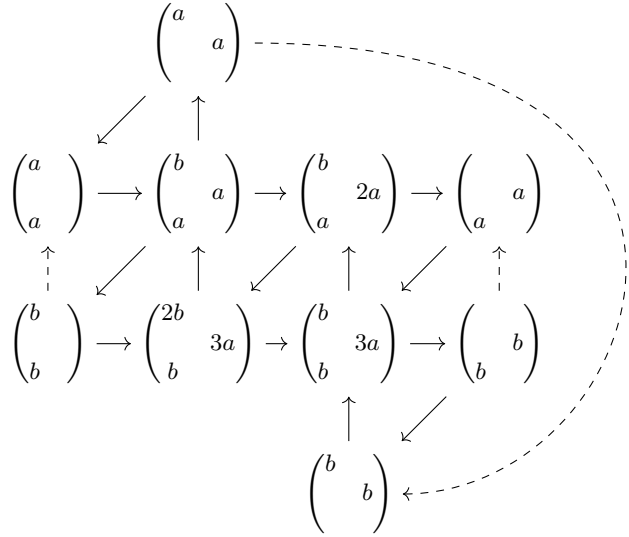
  Figure 14c. The functions for the cluster algebra on $\Conf_3 \A_G$ after performing the fourth and final mutation.
  };

\end{tikzpicture}
\end{center}

\subsection{The second transposition}

We now exhibit the sequence of mutations associated to the transposition $(23) \in S_3$. This transposition is realized by the mutation sequence 

\begin{equation} \label{23}
\begin{gathered}
x_{b1}, \\
x_{b2}, x_{a2}, \\
x_{b1} \\
\end{gathered}
\end{equation}

We view the sequence of mutations as consisting of three stages, corresponding to the three rows in which we listed the mutations. The quiver and the functions transform as follows:

\begin{center}
\begin{tikzpicture}[scale=2]

  \node (xa0) at (-1.5,0) {$\tcfr{a}{a}{}$};
  \node (xa1) at (-0.5,0) {$\tcfr{a}{b}{a}$};
  \node (xa2) at (0.5,0) {$\tcfr{a}{b}{2a}$};
  \node (xa3) at (1.5,0) {$\tcfr{a}{}{a}$};

  \node (xb0) at (-1.5,-1) {$\tcfr{b}{b}{}$};
  \node (xb1) at (-0.5,-1) {$\tcfr{3a}{2b}{3a}$};
  \node (xb2) at (0.5,-1) {$\tcfr{b}{b}{3a}$};
  \node (xb3) at (1.5,-1) {$\tcfr{b}{}{b}$};

  \node (xa) at (-0.5,1) {$\tcfr{}{a}{a}$};
  \node (xb) at (0.5,-2) {$\tcfr{}{b}{b}$};

  \draw [->] (xa1) to (xa0);
  \draw [<-] ([yshift=1.5pt, xshift=1pt] xa2.west) to ([yshift=1.5pt, xshift=-1pt] xa1.east);
  \draw [<-] ([yshift=-1.5pt, xshift=1pt] xa2.west) to ([yshift=-1.5pt, xshift=-1pt] xa1.east);
  \draw [->] (xa3) to (xa2);

  \draw [<-] (xb1) to (xb0);
  \draw [<-] (xb2) to (xb1);
  \draw [->] (xb3) to (xb2);
  \draw [->] (xb2) .. controls +(225:1) and +(315:1) .. (xb0);

  \draw [<-] (xb1) to (xa2);
  \draw [->] (xb2) to (xa3);
  \draw [->] (xa0) to (xa);
  \draw [->] (xb) to (xb3);

  \draw [->] (xb2) to (xb);
  \draw [->] (xa) to (xa1);

  \draw [->, dashed] (xb) .. controls +(0:2) and +(0:4) ..  (xa);

  \draw [->, dashed] (xa0) to (xb0);
  \draw [<-] (xa1) to (xb1);
  \draw [->, dashed] (xa3) to (xb3);

\draw[yshift=-2.5cm]
  node[below,text width=6cm] 
  {
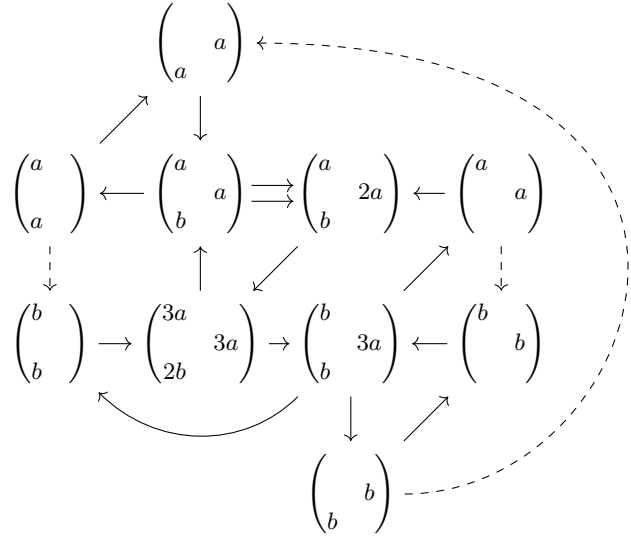
  Figure 15a. The functions for the cluster algebra on $\Conf_3 \A_G$ after performing the first mutation.
  };

\end{tikzpicture}
\end{center}

\begin{center}
\begin{tikzpicture}[scale=2]

  \node (xa0) at (-1.5,0) {$\tcfr{}{a}{a}$};
  \node (xa1) at (-0.5,0) {$\tcfr{b}{b}{a}$};
  \node (xa2) at (0.5,0) {$\tcfr{2a}{b}{a}$};
  \node (xa3) at (1.5,0) {$\tcfr{a}{}{a}$};

  \node (xb0) at (-1.5,-1) {$\tcfr{}{b}{b}$};
  \node (xb1) at (-0.5,-1) {$\tcfr{3a}{2b}{3a}$};
  \node (xb2) at (0.5,-1) {$\tcfr{3a}{b}{b}$};
  \node (xb3) at (1.5,-1) {$\tcfr{b}{}{b}$};

  \node (xa) at (-0.5,1) {$\tcfr{a}{a}{}$};
  \node (xb) at (0.5,-2) {$\tcfr{b}{b}{}$};

  \draw [<-] (xa1) to (xa0);
  \draw [->] ([yshift=1.5pt, xshift=1pt] xa2.west) to ([yshift=1.5pt, xshift=-1pt] xa1.east);
  \draw [->] ([yshift=-1.5pt, xshift=1pt] xa2.west) to ([yshift=-1.5pt, xshift=-1pt] xa1.east);
  \draw [<-] (xa3) to (xa2);

  \draw [->] (xb1) to (xb0);
  \draw [->] (xb2) to (xb1);
  \draw [<-] (xb3) to (xb2);
  \draw [<-] (xb2) .. controls +(225:1) and +(315:1) .. (xb0);

  \draw [->] (xb1) to (xa2);
  \draw [<-] (xb2) to (xa3);
  \draw [<-] (xa0) to (xa);
  \draw [<-] (xb) to (xb3);

  \draw [<-] (xb2) to (xb);
  \draw [<-] (xa) to (xa1);

  \draw [<-, dashed] (xb) .. controls +(0:2) and +(0:4) ..  (xa);

  \draw [<-, dashed] (xa0) to (xb0);
  \draw [->] (xa1) to (xb1);
  \draw [<-, dashed] (xa3) to (xb3);

\draw[yshift=-2.5cm]
  node[below,text width=6cm] 
  {
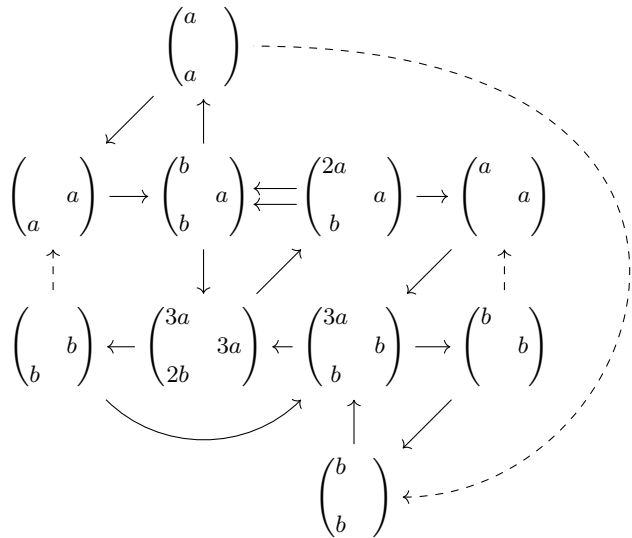
  Figure 15b. The functions for the cluster algebra on $\Conf_3 \A_G$ after performing the third mutation.
  };

\end{tikzpicture}
\end{center}

\begin{center}
\begin{tikzpicture}[scale=2]

  \node (xa0) at (-1.5,0) {$\tcfr{}{a}{a}$};
  \node (xa1) at (-0.5,0) {$\tcfr{a}{b}{a}$};
  \node (xa2) at (0.5,0) {$\tcfr{2a}{b}{a}$};
  \node (xa3) at (1.5,0) {$\tcfr{a}{}{a}$};

  \node (xb0) at (-1.5,-1) {$\tcfr{}{b}{b}$};
  \node (xb1) at (-0.5,-1) {$\tcfr{3a}{2b}{b}$};
  \node (xb2) at (0.5,-1) {$\tcfr{3a}{b}{b}$};
  \node (xb3) at (1.5,-1) {$\tcfr{b}{}{b}$};

  \node (xa) at (-0.5,1) {$\tcfr{a}{a}{}$};
  \node (xb) at (0.5,-2) {$\tcfr{b}{b}{}$};

  \draw [->] (xa0) to (xa1);
  \draw [->] (xa1) to (xa2);
  \draw [->] (xa2) to  (xa3);

  \draw [->] (xa1) to  (xa);
  \draw [->] (xa) to (xa0);

  \draw [->] (xb0) to (xb1);
  \draw [->] (xb1) to (xb2);
  \draw [->] (xb) to (xb2);

  \draw [->] (xa1) to (xb0);
  \draw [->] (xa2) to (xb1);
  \draw [->] (xa3) to (xb2);
  \draw [->] (xb3) to (xb);

  \draw [->] (xb2) to (xb3);

  \draw [->, dashed] (xb3) to (xa3);

  \draw [->, dashed] (xb0) to (xa0);
  \draw [->] (xb1) to (xa1);
  \draw [->] (xb2) to (xa2);
  \draw [->, dashed] (xa) .. controls +(0:4) and +(0:2) ..   (xb);

\draw[yshift=-2.5cm]
  node[below,text width=6cm] 
  {
  Figure 15c. The functions for the cluster algebra on $\Conf_3 \A_G$ after performing the fourth and final mutation.
  };

\end{tikzpicture}
\end{center}

Note that the sequence of mutations associated to a transposition results in a seed isomorphic to the inital one, except with all the arrows \emph{reversed}.

\begin{rmk} In order to show that the $S_3$ symmetries we have realized are compatible with the descriptions given in Section \ref{S_3}, one should check that the sequence of mutations corresponding to $(13)(23)$ realizes the cyclic shift $T$ and that the sequence of mutations corresponding to $(13)(23)(13)$ realizes the transposition $(12)$ as described there. We discuss the latter in the next section.
\end{rmk}

\subsection{The third transposition, Langlands duality}

Let us first observe that the third transposition $(12) \in S_3$ can be realized by the composition of transpositions $(13)(23)(13)$. Thus the action of this transposition on $\Conf_3 \A_G$ can be realized via the sequence of mutations

\begin{equation} \label{23}
\begin{gathered}
x_{a2}, x_{a1}, x_{b1}, x_{a2} \\
x_{b1}, x_{b2}, x_{a2}, x_{b1} \\
x_{a2}, x_{a1}, x_{b1}, x_{a2} 
\end{gathered}
\end{equation}

On the other hand, recall that the third transposition relates the clusters attached to the reduced words $s_bs_as_bs_as_bs_a$ and $s_as_bs_as_bs_as_b$. Thus it can be realized as a sequence of mutations using the work of \cite{BFZ}, and \cite{BZ}, where this sequence first appeared. Thus we have given an alternative proof of the following theorem:

\begin{theorem} \cite{BZ} There exists a sequence of mutations relating the reduced words $s_bs_as_bs_as_bs_a$ and $s_as_bs_as_bs_as_b$.
\end{theorem}

The same result was later rederived in \cite{FG3}. In all both \cite{BZ} and \cite{FG3}, the computation of the sequence comes down to an involved, and not very illuminating, computation. Our analysis, then, gives a more conceptual proof of the existence of such a sequence of mutations: we relate the sequence to the transposition $(12) \in S_3$ and break down the sequence of mutations into three stages corresponding to the decomposition $(12)=(13)(23)(13)$.

This suggests, then, that the cluster algebra for  $\Conf_3 \A_G$ is in some sense \emph{more fundamental} than the one for $B^-_G$. Note that the transpositions $(13)$ and $(23)$ take us outside the world of reduced word decompositions in order to prove a fundamental fact about relating reduced word decompositions.

An analogous result is also true in the case of where $G$ has type $B_2=C_2$, i.e., when our group is $Spin_5=Sp_4$. As explained in \cite{Le}, the transpositions $(13)$ and $(23)$ for the group $Sp_4$ each consist of one mutation. They compose to give the transposition $(12)=(13)(23)(13)$ in a sequence of three mutations. This sequence of mutations relates the reduced words $s_1s_2s_1s_2$ and $s_2s_1s_2s_1$.

Finally, we wish to say some words about Langlands duality. The group $G_2$ is self-dual. The duality interchanges the simple roots $\alpha$ and $\beta$. This manifests itself in a symmetry for the quiver we have constructed for $\Conf_3 \A_G$: if we rotate the quiver by $180$ degrees, switch the coloring of all the vertices, and reverse all the arrows, we obtain the same quiver again.

Moreover, this symmetry is stable under mutation. Let us be more precise. Let $L$ be the transformation of the quiver which reverses all arrows, switches the coloring of all vertices, and interchanges the following pairs of vertices of the quiver:
$$x_{a} \longleftrightarrow x_{b}$$
$$x_{a0} \longleftrightarrow x_{b3}$$
$$x_{a1} \longleftrightarrow x_{b2}$$
$$x_{a2} \longleftrightarrow x_{b1}$$
$$x_{a3} \longleftrightarrow x_{b0}$$
Then consider any sequence of mutations, $x_{i_1}, \dots, x_{i_k}$. Then mutating $x_{i_1}, \dots, x_{i_k}$ and then applying $L$ gives the same result as applying $L$ and then mutating $L(x_{i_1}), \dots, L(x_{i_k})$.

The reason this happens is that the cluster algebra on $\Conf_3 \A_{G}$ is Langlands dual to itself:

\begin{definition} \cite{FG2} Two seeds that have the same set of vertices $I$ and the same set of frozen vertices $I_0$ are said to be {\em Langlands dual} if they have $B$-matrices $b_{ij}$ and $b^{\vee}_{ij}$ and multipliers $d_i$ and $d^{\vee}_i$ where
$$d_i=(d^{\vee}_i)^{-1}D,$$
$$b_{ij}d_j=-b^{\vee}_{ij}d^{\vee}_j,$$
for some rational number $D$.
\end{definition}

\begin{rmk} Note that the multipliers $d_i$ for a cluster algebra are determined only up to simultaneous scaling by a rational number. Conventions sometimes differ on how to specify the values for the $d_i$. This is one reason the rational number $D$ appears in the above definition.
\end{rmk}

In other words, Langlands duality on seeds involves switching colors of vertices and reversing all arrows. This procedure applied to $\Conf_3 \A_{G}$ gives the same seed with the vertices rearranged. Amalgamation gives the same result for $\Conf_m \A_{G}$.

Up until now, we have only discussed Langlands duality on the level of seeds. Let us discuss what it does on the level of functions, or cluster variables.

When two seeds are Langlands dual, there is a close relationship between the resulting cluster algebras. Suppose that $(I, I_0, b_{ij}, d_i)$ and $(I, I_0, b^{\vee}_{ij}, d^{\vee}_i)$ are Langlands dual seeds. Let the cluster variables for the initial seeds be $f_1, \dots, f_n$ and $f^{\vee}_1, \dots, f^{\vee}_n$, respectively. These cluster variables are naturally in bijection. Then if we mutate $f_k$ to obtain the new cluster variable $f_k'$, we can do the same to $f^{\vee}_k$ to get $(f'_k)^{\vee}$ and then match $f_k'$ and $(f'_k)^{\vee}$. Continuing in this manner, one conjecturally gets a bijection between all the cluster variables for the Langlands dual seeds.

We now explain the relationship between the functions $f_k'$ and $(f'_k)^{\vee}$. We have an isomorphism between the weight spaces and the coweight spaces for the group $G$ given by the Killing form, which takes
$$\alpha \rightarrow \alpha^{\vee} $$
$$\beta \rightarrow 3\beta^{\vee} $$

On the other hand, the isomorphism between $G$ and its Langlands dual means that $\alpha^{\vee}=\beta$ and $\beta^{\vee}=\alpha$. Composing these gives an isomorphism $L$ from the weight space for $G$ to itself which maps (here we describe the map on fundamental weights instead of simple roots)
$$a \rightarrow b$$
$$b \rightarrow 3a$$

We have the following observation:

\begin{observation} Suppose that we have a cluster variable
$$f \in [V_{\lambda} \otimes V_{\mu} \otimes V_{\nu}]^{G} \subset \mathcal{O}(\Conf_3 \A_{G}).$$
Then if $f$ is associated to a black vertex and $f^{\vee}$ is associated to a white vertex, then
$$f^{\vee} \in [V_{L(\lambda)} \otimes V_{L(\mu)} \otimes V_{L(\nu)}]^{G} \subset \mathcal{O}(\Conf_3 \A_{G})$$
while if $f$ is associated to a white vertex and $f^{\vee}$ is associated to a black vertex
$$f^{\vee} \in [V_{L(\lambda)/3} \otimes V_{L(\mu)/3} \otimes V_{L(\nu)/3}]^{G} \subset \mathcal{O}(\Conf_3 \A_{G}).$$
\end{observation}

This is clearly true in the initial cluster, and as long as all the cluster variables are functions on $\Conf_3 \A_{G}$ and therefore given by tensor invariants, and not just rational functions on $\Conf_3 \A_{G}$ (as we expect, but do not know how to prove), it is easy to check that the above observation remains true under mutation. Certainly, in all the clusters we consider in this paper this will be the case. For example, the Langlands dual to the seed in Figure 13 is given by applying the transposition $(12)$ to the original seed for $\Conf_3 \A_{G}$, giving us a third way of looking at this transposition!

One last observation:

\begin{observation} The sequence of mutations for the transpositions $(13)$ and $(23)$ are related by Langlands duality. Moreover, the sequence of mutations of the flip (discussed in the next section) and the reverse of the sequence of mutations for a flip are related by Langlands duality. An analogous statement is true for the group $Sp_4$.
\end{observation}

\subsection{The sequence of mutations for a flip}

In this section, we will give a sequence of mutations that relates two of the clusters coming from different triangulations of the $4$-gon. Combined with the previous section, this allows us to connect by mutations all $72$ different clusters we have constructed for $\Conf_4 \A_{G}$ (There are $2$ possible triangulations, and $6$ clusters for each triangle, giving $6 \cdot 6 \cdot 2$ total clusters.).

Given a configuration $(A,B,C,D) \in \Conf_4 \A_{G}$, we will give a sequence of mutations that relates a cluster coming from the triangulation $ABC, ACD$ to a cluster coming from the triangulation $ABD, BCD$. Refer to Figure 6 for a diagram of the configuration $(A,B,C,D)$.

We will need to relabel the quiver with vertices $x_{ij}$, $y_k$, with $j=a, b$, $-n \leq i \leq n$, and $k =\pm a, \pm b$. The quiver we will start with is as in Figure 16.

\begin{center}
\begin{tikzpicture}[scale=2.4]

  \node (x0a) at (0,0) {\Large $\ontop{x_{0a}}{\bullet}$};
  \node (x1a) at (1,0) {\Large $\ontop{x_{1a}}{\bullet}$};
  \node (x2a) at (2,0) {\Large $\ontop{x_{2a}}{\bullet}$};
  \node (x3a) at (3,0) {\Large $\ontop{x_{3a}}{\bullet}$};
  \node (x-1a) at (-1,0) {\Large $\ontop{x_{-1a}}{\bullet}$};
  \node (x-2a) at (-2,0) {\Large $\ontop{x_{-2a}}{\bullet}$};
  \node (x-3a) at (-3,0) {\Large $\ontop{x_{-3a}}{\bullet}$};

  \node (x0b) at (0,-1) {\Large $\ontop{x_{0b}}{\circ}$};
  \node (x1b) at (1,-1) {\Large $\ontop{x_{1b}}{\circ}$};
  \node (x2b) at (2,-1) {\Large $\ontop{x_{2b}}{\circ}$};
  \node (x3b) at (3,-1) {\Large $\ontop{x_{3b}}{\circ}$};
  \node (x-1b) at (-1,-1) {\Large $\ontop{x_{-1b}}{\circ}$};
  \node (x-2b) at (-2,-1) {\Large $\ontop{x_{-2b}}{\circ}$};
  \node (x-3b) at (-3,-1) {\Large $\ontop{x_{-3b}}{\circ}$};

  \node (ya) at (0.5,1) {\Large $\ontop{y_{a}}{\bullet}$};
  \node (yb) at (2.5,-2) {\Large $\ontop{y_{b}}{\circ}$};
  \node (y-a) at (-0.5,1) {\Large $\ontop{y_{-a}}{\bullet}$};
  \node (y-b) at (-2.5,-2) {\Large $\ontop{y_{-b}}{\circ}$};

  \draw [->] (x1a) to (x0a);
  \draw [->] (x2a) to (x1a);
  \draw [->] (x3a) to (x2a);

  \draw [->] (x1b) to (x0b);
  \draw [->] (x2b) to (x1b);
  \draw [->] (x3b) to (x2b);

  \draw [->] (x0b) to (x1a);
  \draw [->] (x1b) to (x2a);
  \draw [->] (x2b) to (x3a);
  \draw [->] (x0a) to (ya);
  \draw [->] (yb) to (x3b);

  \draw [->] (x2b) to (yb);
  \draw [->] (ya) to (x1a);


  \draw [->] (x0a) to (x0b);
  \draw [->] (x1a) to (x1b);
  \draw [->] (x2a) to (x2b);
  \draw [->, dashed] (x3a) to (x3b);

  \draw [->] (x-1a) to (x0a);
  \draw [->] (x-2a) to (x-1a);
  \draw [->] (x-3a) to (x-2a);

  \draw [->] (x-1b) to (x0b);
  \draw [->] (x-2b) to (x-1b);
  \draw [->] (x-3b) to (x-2b);

  \draw [->] (x0b) to (x-1a);
  \draw [->] (x-1b) to (x-2a);
  \draw [->] (x-2b) to (x-3a);
  \draw [->] (x0a) to (y-a);
  \draw [->] (y-b) to (x-3b);

  \draw [->] (x-2b) to (y-b);
  \draw [->] (y-a) to (x-1a);


  \draw [->] (x-1a) to (x-1b);
  \draw [->] (x-2a) to (x-2b);
  \draw [->, dashed] (x-3a) to (x-3b);

\draw[yshift=-1.85cm]
  node[below,text width=6cm] 
  {
  Figure 16. The quiver for the cluster algebra on $\Conf_4 \A_{G}$.
  };

\end{tikzpicture}
\end{center}

We have removed the dashed arrows between $y_b$ and $y_a$ and between $y_{-b}$ and $y_{-a}$ for simplicity. These arrows neither change nor play any role during our sequence of mutations. We write down the functions associated to the vertices in the quiver in Figure 17.

\begin{center}
\begin{tikzpicture}[scale=2.4]

  \node (x0a) at (0,0) {$\dud{a}{a}$};
  \node (x1a) at (1,0) {$\tcfr{a}{b}{a}$};
  \node (x2a) at (2,0) {$\tcfr{a}{b}{2a}$};
  \node (x3a) at (3,0) {$\tcfr{a}{}{a}$};
  \node (x-1a) at (-1,0) {$\tcfl{b}{a}{a}$};
  \node (x-2a) at (-2,0) {$\tcfl{b}{2a}{a}$};
  \node (x-3a) at (-3,0) {$\tcfl{}{a}{a}$};

  \node (x0b) at (0,-1) {$\dud{b}{b}$};
  \node (x1b) at (1,-1) {$\tcfr{b}{2b}{3a}$};
  \node (x2b) at (2,-1) {$\tcfr{b}{b}{3a}$};
  \node (x3b) at (3,-1) {$\tcfr{b}{}{b}$};
  \node (x-1b) at (-1,-1) {$\tcfl{2b}{3a}{b}$};
  \node (x-2b) at (-2,-1) {$\tcfl{b}{3a}{b}$};
  \node (x-3b) at (-3,-1) {$\tcfl{}{b}{b}$};

  \node (ya) at (0.5,1) {$\tcfr{}{a}{a}$};
  \node (yb) at (2.5,-2) {$\tcfr{}{b}{b}$};
  \node (y-a) at (-0.5,1) {$\tcfl{a}{a}{}$};
  \node (y-b) at (-2.5,-2) {$\tcfl{b}{b}{}$};

  \draw [->] (x1a) to (x0a);
  \draw [->] (x2a) to (x1a);
  \draw [->] (x3a) to (x2a);

  \draw [->] (x1b) to (x0b);
  \draw [->] (x2b) to (x1b);
  \draw [->] (x3b) to (x2b);

  \draw [->] (x0b) to (x1a);
  \draw [->] (x1b) to (x2a);
  \draw [->] (x2b) to (x3a);
  \draw [->] (x0a) to (ya);
  \draw [->] (yb) to (x3b);

  \draw [->] (x2b) to (yb);
  \draw [->] (ya) to (x1a);

  \draw [->] (x0a) to (x0b);
  \draw [->] (x1a) to (x1b);
  \draw [->] (x2a) to (x2b);
  \draw [->, dashed] (x3a) to (x3b);

  \draw [->] (x-1a) to (x0a);
  \draw [->] (x-2a) to (x-1a);
  \draw [->] (x-3a) to (x-2a);

  \draw [->] (x-1b) to (x0b);
  \draw [->] (x-2b) to (x-1b);
  \draw [->] (x-3b) to (x-2b);

  \draw [->] (x0b) to (x-1a);
  \draw [->] (x-1b) to (x-2a);
  \draw [->] (x-2b) to (x-3a);
  \draw [->] (x0a) to (y-a);
  \draw [->] (y-b) to (x-3b);

  \draw [->] (x-2b) to (y-b);
  \draw [->] (y-a) to (x-1a);

  \draw [->] (x-1a) to (x-1b);
  \draw [->] (x-2a) to (x-2b);
  \draw [->, dashed] (x-3a) to (x-3b);

\draw[yshift=-1.85cm]
  node[below,text width=6cm] 
  {
  Figure 17. The functions for the cluster algebra on $\Conf_4 \A_{G}$.
  };

\end{tikzpicture}
\end{center}

The sequence of mutations for the flip of a triangulation is as follows:

\begin{equation} \label{flip}
\begin{gathered}
x_{0a}, \\
x_{-1a}, x_{0b}, x_{1a}, \\
x_{-2a}, x_{-1b}, x_{0a}, x_{1b}, x_{2a}, \\
x_{-2b}, x_{-1a}, x_{0b}, x_{1a}, x_{2b}, \\
x_{-1b}, x_{0a}, x_{1b}, \\
x_{0b}, \\
 \end{gathered}
\end{equation}

The rows above correspond to what we will call the six stages of the mutation sequence. In each stage, the vertices may be mutated in any order. We depict how the quiver and the functions change after each stage of the sequence of mutations in Figure 18.

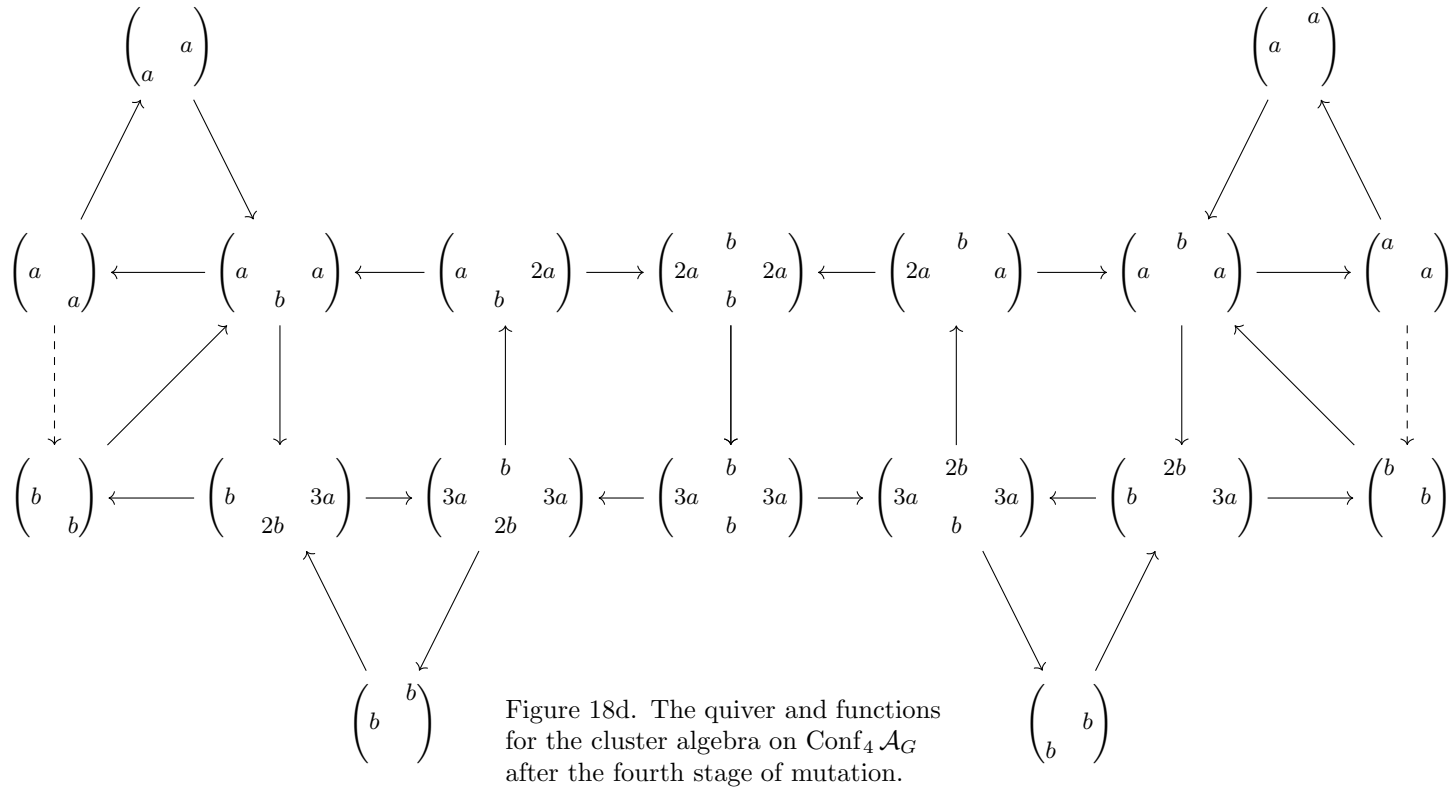
\begin{center}
\begin{tikzpicture}[scale=2.4]

  \node (x0a) at (0,0) {$\qcf{b}{a}{b}{a}$};
  \node (x1a) at (1,0) {$\tcfr{a}{b}{a}$};
  \node (x2a) at (2,0) {$\tcfr{a}{b}{2a}$};
  \node (x3a) at (3,0) {$\tcfr{a}{}{a}$};
  \node (x-1a) at (-1,0) {$\tcfl{b}{a}{a}$};
  \node (x-2a) at (-2,0) {$\tcfl{b}{2a}{a}$};
  \node (x-3a) at (-3,0) {$\tcfl{}{a}{a}$};

  \node (x0b) at (0,-1) {$\dud{b}{b}$};
  \node (x1b) at (1,-1) {$\tcfr{b}{2b}{3a}$};
  \node (x2b) at (2,-1) {$\tcfr{b}{b}{3a}$};
  \node (x3b) at (3,-1) {$\tcfr{b}{}{b}$};
  \node (x-1b) at (-1,-1) {$\tcfl{2b}{3a}{b}$};
  \node (x-2b) at (-2,-1) {$\tcfl{b}{3a}{b}$};
  \node (x-3b) at (-3,-1) {$\tcfl{}{b}{b}$};

  \node (ya) at (-0.5,1) {$\tcfr{}{a}{a}$};
  \node (yb) at (2.5,-2) {$\tcfr{}{b}{b}$};
  \node (y-a) at (0.5,1) {$\tcfl{a}{a}{}$};
  \node (y-b) at (-2.5,-2) {$\tcfl{b}{b}{}$};

  \draw [<-] (x1a) to (x0a);
  \draw [->] (x2a) to (x1a);
  \draw [->] (x3a) to (x2a);

  \draw [->] (x1b) to (x0b);
  \draw [->] (x2b) to (x1b);
  \draw [->] (x3b) to (x2b);

  \draw [->] (x1b) to (x2a);
  \draw [->] (x2b) to (x3a);

  \draw [<-] (x0a) to (x0b);
  \draw [->] (x1a) to (x1b);
  \draw [->] (x2a) to (x2b);
  \draw [->, dashed] (x3a) to (x3b);

  \draw [<-] (x-1a) to (x0a);
  \draw [->] (x-2a) to (x-1a);
  \draw [->] (x-3a) to (x-2a);

  \draw [->] (x-1b) to (x0b);
  \draw [->] (x-2b) to (x-1b);
  \draw [->] (x-3b) to (x-2b);

  \draw [->] (x-1b) to (x-2a);
  \draw [->] (x-2b) to (x-3a);

  \draw [<-] (x0a) to (x0b);
  \draw [->] (x-1a) to (x-1b);
  \draw [->] (x-2a) to (x-2b);
  \draw [->, dashed] (x-3a) to (x-3b);

  \draw [<-] (x0a) to (y-a);
  \draw [<-] (y-a) to (x1a);

  \draw [->] (yb) to (x3b);
  \draw [->] (x2b) to (yb);

  \draw [<-] (x0a) to (ya);
  \draw [<-] (ya) to (x-1a);

  \draw [->] (y-b) to (x-3b);
  \draw [->] (x-2b) to (y-b);

\draw[yshift=-1.85cm]
  node[below,text width=6cm] 
  {
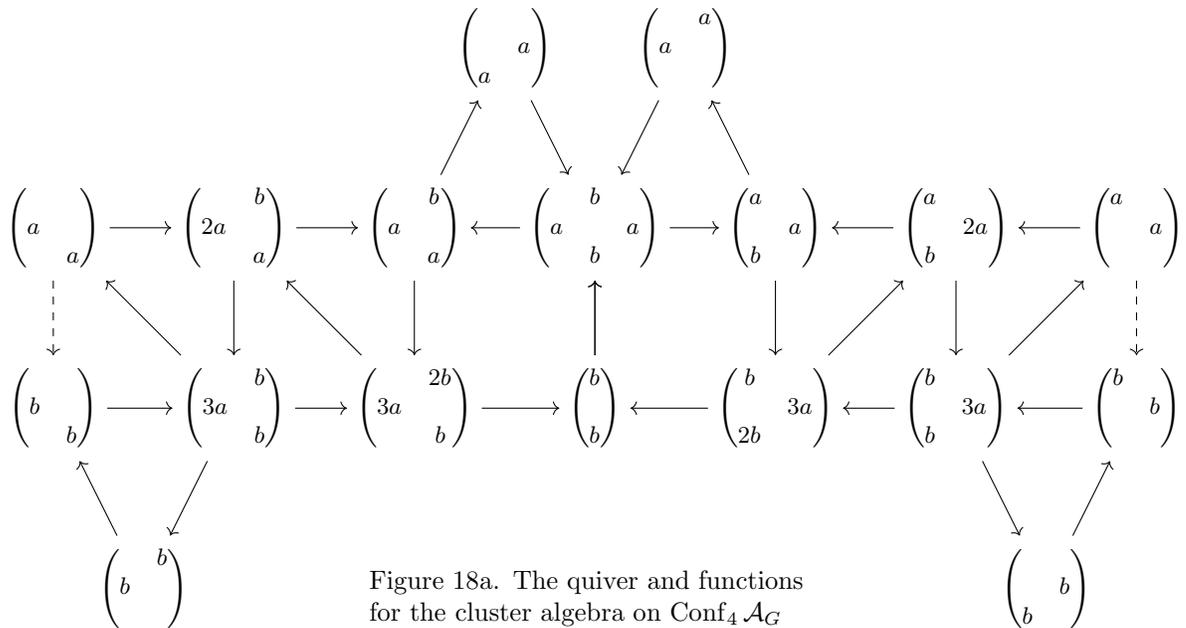
  Figure 18a. The quiver and functions for the cluster algebra on $\Conf_4 \A_{G}$ after the first stage of mutation.
  };

\end{tikzpicture}
\end{center}

\begin{center}
\begin{tikzpicture}[scale=2.4]

  \node (x0a) at (0,0) {$\qcf{b}{a}{b}{a}$};
  \node (x1a) at (1,0) {$\qcf{b}{a}{b}{2a}$};
  \node (x2a) at (2,0) {$\tcfr{a}{b}{2a}$};
  \node (x3a) at (3,0) {$\tcfr{a}{}{a}$};
  \node (x-1a) at (-1,0) {$\qcf{b}{2a}{b}{a}$};
  \node (x-2a) at (-2,0) {$\tcfl{b}{2a}{a}$};
  \node (x-3a) at (-3,0) {$\tcfl{}{a}{a}$};

  \node (x0b) at (0,-1) {$\qcf{2b}{3a}{2b}{3a}$};
  \node (x1b) at (1,-1) {$\tcfr{b}{2b}{3a}$};
  \node (x2b) at (2,-1) {$\tcfr{b}{b}{3a}$};
  \node (x3b) at (3,-1) {$\tcfr{b}{}{b}$};
  \node (x-1b) at (-1,-1) {$\tcfl{2b}{3a}{b}$};
  \node (x-2b) at (-2,-1) {$\tcfl{b}{3a}{b}$};
  \node (x-3b) at (-3,-1) {$\tcfl{}{b}{b}$};

  \node (ya) at (-1.5,1) {$\tcfr{}{a}{a}$};
  \node (yb) at (2.5,-2) {$\tcfr{}{b}{b}$};
  \node (y-a) at (1.5,1) {$\tcfl{a}{a}{}$};
  \node (y-b) at (-2.5,-2) {$\tcfl{b}{b}{}$};

  \draw [->] (x1a) to (x0a);
  \draw [<-] (x2a) to (x1a);
  \draw [->] (x3a) to (x2a);

  \draw [<-] (x1b) to (x0b);
  \draw [->] (x2b) to (x1b);
  \draw [->] (x3b) to (x2b);

  \draw [->] (x2b) to (x3a);

  \draw [->] (x0a) to (x0b);
  \draw [<-] (x1a) to (x1b);
  \draw [->] (x2a) to (x2b);
  \draw [->, dashed] (x3a) to (x3b);

  \draw [->] (x-1a) to (x0a);
  \draw [<-] (x-2a) to (x-1a);
  \draw [->] (x-3a) to (x-2a);

  \draw [<-] (x-1b) to (x0b);
  \draw [->] (x-2b) to (x-1b);
  \draw [->] (x-3b) to (x-2b);

  \draw [->] (x-2b) to (x-3a);

  \draw [->] (x0a) to (x0b);
  \draw [<-] (x-1a) to (x-1b);
  \draw [->] (x-2a) to (x-2b);
  \draw [->, dashed] (x-3a) to (x-3b);

  \draw [<-] (x1a) to (y-a);
  \draw [<-] (y-a) to (x2a);

  \draw [->] (yb) to (x3b);
  \draw [->] (x2b) to (yb);

  \draw [<-] (x-1a) to (ya);
  \draw [<-] (ya) to (x-2a);

  \draw [->] (y-b) to (x-3b);
  \draw [->] (x-2b) to (y-b);

\draw[yshift=-1.85cm]
  node[below,text width=6cm] 
  {
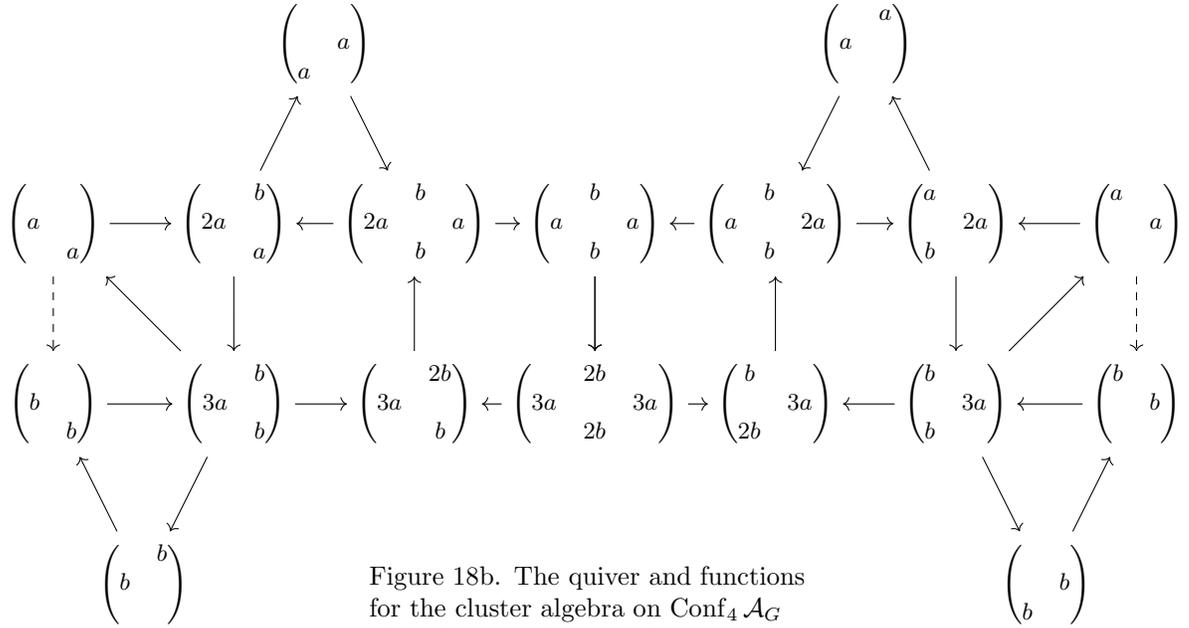
  Figure 18b. The quiver and functions for the cluster algebra on $\Conf_4 \A_{G}$ after the second stage of mutation.
  };

\end{tikzpicture}
\end{center}

\begin{center}
\begin{tikzpicture}[scale=3,rotate=90]

  \node (x0a) at (0,0) [rotate=90] {$\qcf{b}{2a}{b}{2a}$};
  \node (x1a) at (1,0) [rotate=90] {$\qcf{b}{a}{b}{2a}$};
  \node (x2a) at (2,0) [rotate=90] {$\tcfu{b}{a}{a}$};
  \node (x3a) at (3,0) [rotate=90] {$\tcfr{a}{}{a}$};
  \node (x-1a) at (-1,0) [rotate=90] {$\qcf{b}{2a}{b}{a}$};
  \node (x-2a) at (-2,0) [rotate=90] {$\tcfd{a}{b}{a}$};
  \node (x-3a) at (-3,0) [rotate=90] {$\tcfl{}{a}{a}$};

  \node (x0b) at (0,-1) [rotate=90] {$\qcf{2b}{3a}{2b}{3a}$};
  \node (x1b) at (1,-1) [rotate=90] {$\qcf{2b}{3a}{b}{3a}$};
  \node (x2b) at (2,-1) [rotate=90] {$\tcfr{b}{b}{3a}$};
  \node (x3b) at (3,-1) [rotate=90] {$\tcfr{b}{}{b}$};
  \node (x-1b) at (-1,-1) [rotate=90] {$\qcf{b}{3a}{2b}{3a}$};
  \node (x-2b) at (-2,-1) [rotate=90] {$\tcfl{b}{3a}{b}$};
  \node (x-3b) at (-3,-1) [rotate=90] {$\tcfl{}{b}{b}$};

  \node (ya) at (-2.5,1) [rotate=90] {$\tcfr{}{a}{a}$};
  \node (yb) at (2.5,-2) [rotate=90] {$\tcfr{}{b}{b}$};
  \node (y-a) at (2.5,1) [rotate=90] {$\tcfl{a}{a}{}$};
  \node (y-b) at (-2.5,-2) [rotate=90] {$\tcfl{b}{b}{}$};

  \draw [<-] (x1a) to (x0a);
  \draw [->] (x2a) to (x1a);
  \draw [<-] (x3a) to (x2a);

  \draw [->] (x1b) to (x0b);
  \draw [<-] (x2b) to (x1b);
  \draw [->] (x3b) to (x2b);


  \draw [<-] (x0a) to (x0b);
  \draw [->] (x1a) to (x1b);
  \draw [<-] (x2a) to (x2b);
  \draw [->, dashed] (x3a) to (x3b);

  \draw [<-] (x-1a) to (x0a);
  \draw [->] (x-2a) to (x-1a);
  \draw [<-] (x-3a) to (x-2a);

  \draw [->] (x-1b) to (x0b);
  \draw [<-] (x-2b) to (x-1b);
  \draw [->] (x-3b) to (x-2b);


  \draw [<-] (x0a) to (x0b);
  \draw [->] (x-1a) to (x-1b);
  \draw [<-] (x-2a) to (x-2b);
  \draw [->, dashed] (x-3a) to (x-3b);

  \draw [<-] (x2a) to (y-a);
  \draw [<-] (y-a) to (x3a);

  \draw [->] (yb) to (x3b);
  \draw [->] (x2b) to (yb);

  \draw [<-] (x-2a) to (ya);
  \draw [<-] (ya) to (x-3a);

  \draw [->] (y-b) to (x-3b);
  \draw [->] (x-2b) to (y-b);

\draw[yshift=-1.85cm]
  node[below,text width=6cm,rotate=90] 
  {
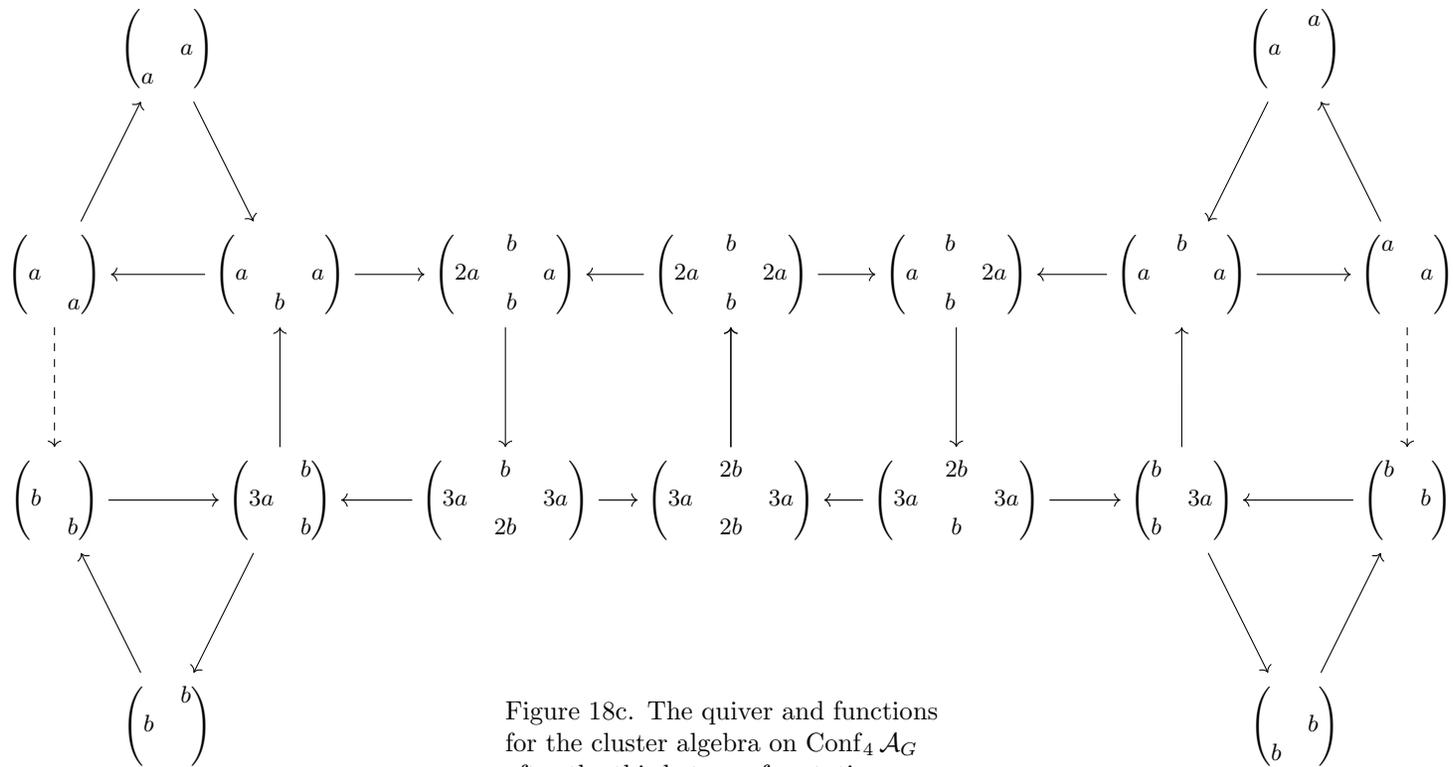
  Figure 18c. The quiver and functions for the cluster algebra on $\Conf_4 \A_{G}$ after the third stage of mutation.
  };

\end{tikzpicture}
\end{center}

\begin{center}
\begin{tikzpicture}[scale=3,rotate=90]

  \node (x0a) at (0,0) [rotate=90] {$\qcf{b}{2a}{b}{2a}$};
  \node (x1a) at (1,0) [rotate=90] {$\tcfu{b}{2a}{a}$};
  \node (x2a) at (2,0) [rotate=90] {$\tcfu{b}{a}{a}$};
  \node (x3a) at (3,0) [rotate=90] {$\tcfr{a}{}{a}$};
  \node (x-1a) at (-1,0) [rotate=90] {$\tcfd{a}{b}{2a}$};
  \node (x-2a) at (-2,0) [rotate=90] {$\tcfd{a}{b}{a}$};
  \node (x-3a) at (-3,0) [rotate=90] {$\tcfl{}{a}{a}$};

  \node (x0b) at (0,-1) [rotate=90] {$\qcf{b}{3a}{b}{3a}$};
  \node (x1b) at (1,-1) [rotate=90] {$\qcf{2b}{3a}{b}{3a}$};
  \node (x2b) at (2,-1) [rotate=90] {$\tcfu{2b}{b}{3a}$};
  \node (x3b) at (3,-1) [rotate=90] {$\tcfr{b}{}{b}$};
  \node (x-1b) at (-1,-1) [rotate=90] {$\qcf{b}{3a}{2b}{3a}$};
  \node (x-2b) at (-2,-1) [rotate=90] {$\tcfd{b}{2b}{3a}$};
  \node (x-3b) at (-3,-1) [rotate=90] {$\tcfl{}{b}{b}$};

  \node (ya) at (-2.5,1) [rotate=90] {$\tcfr{}{a}{a}$};
  \node (yb) at (1.5,-2) [rotate=90] {$\tcfr{}{b}{b}$};
  \node (y-a) at (2.5,1) [rotate=90] {$\tcfl{a}{a}{}$};
  \node (y-b) at (-1.5,-2) [rotate=90] {$\tcfl{b}{b}{}$};

  \draw [->] (x1a) to (x0a);
  \draw [<-] (x2a) to (x1a);
  \draw [<-] (x3a) to (x2a);

  \draw [<-] (x1b) to (x0b);
  \draw [->] (x2b) to (x1b);
  \draw [<-] (x3b) to (x2b);

  \draw [->] (x3b) to (x2a);

  \draw [->] (x0a) to (x0b);
  \draw [<-] (x1a) to (x1b);
  \draw [->] (x2a) to (x2b);
  \draw [->, dashed] (x3a) to (x3b);

  \draw [->] (x-1a) to (x0a);
  \draw [<-] (x-2a) to (x-1a);
  \draw [<-] (x-3a) to (x-2a);

  \draw [<-] (x-1b) to (x0b);
  \draw [->] (x-2b) to (x-1b);
  \draw [<-] (x-3b) to (x-2b);

  \draw [->] (x-3b) to (x-2a);

  \draw [->] (x0a) to (x0b);
  \draw [<-] (x-1a) to (x-1b);
  \draw [->] (x-2a) to (x-2b);
  \draw [->, dashed] (x-3a) to (x-3b);

  \draw [<-] (x2a) to (y-a);
  \draw [<-] (y-a) to (x3a);

  \draw [->] (yb) to (x2b);
  \draw [->] (x1b) to (yb);

  \draw [<-] (x-2a) to (ya);
  \draw [<-] (ya) to (x-3a);

  \draw [->] (y-b) to (x-2b);
  \draw [->] (x-1b) to (y-b);

\draw[yshift=-1.85cm]
  node[below,text width=6cm,rotate=90] 
  {
  Figure 18d. The quiver and functions for the cluster algebra on $\Conf_4 \A_{G}$ after the fourth stage of mutation.
  };

\end{tikzpicture}
\end{center}

\begin{center}
\begin{tikzpicture}[scale=2.4]

  \node (x0a) at (0,0) {$\dlr{a}{a}$};
  \node (x1a) at (1,0) {$\tcfu{b}{2a}{a}$};
  \node (x2a) at (2,0) {$\tcfu{b}{a}{a}$};
  \node (x3a) at (3,0) {$\tcfr{a}{}{a}$};
  \node (x-1a) at (-1,0) {$\tcfd{a}{b}{2a}$};
  \node (x-2a) at (-2,0) {$\tcfd{a}{b}{a}$};
  \node (x-3a) at (-3,0) {$\tcfl{}{a}{a}$};

  \node (x0b) at (0,-1) {$\qcf{b}{3a}{b}{3a}$};
  \node (x1b) at (1,-1) {$\tcfu{b}{b}{3a}$};
  \node (x2b) at (2,-1) {$\tcfu{2b}{b}{3a}$};
  \node (x3b) at (3,-1) {$\tcfr{b}{}{b}$};
  \node (x-1b) at (-1,-1) {$\tcfd{b}{b}{3a}$};
  \node (x-2b) at (-2,-1) {$\tcfd{b}{2b}{3a}$};
  \node (x-3b) at (-3,-1) {$\tcfl{}{b}{b}$};

  \node (ya) at (-2.5,1) {$\tcfr{}{a}{a}$};
  \node (yb) at (0.5,-2) {$\tcfr{}{b}{b}$};
  \node (y-a) at (2.5,1) {$\tcfl{a}{a}{}$};
  \node (y-b) at (-0.5,-2) {$\tcfl{b}{b}{}$};

  \draw [<-] (x1a) to (x0a);
  \draw [<-] (x2a) to (x1a);
  \draw [<-] (x3a) to (x2a);

  \draw [->] (x1b) to (x0b);
  \draw [<-] (x2b) to (x1b);
  \draw [<-] (x3b) to (x2b);

  \draw [->] (x2b) to (x1a);
  \draw [->] (x3b) to (x2a);

  \draw [<-] (x0a) to (x0b);
  \draw [->] (x1a) to (x1b);
  \draw [->] (x2a) to (x2b);
  \draw [->, dashed] (x3a) to (x3b);

  \draw [<-] (x-1a) to (x0a);
  \draw [<-] (x-2a) to (x-1a);
  \draw [<-] (x-3a) to (x-2a);

  \draw [->] (x-1b) to (x0b);
  \draw [<-] (x-2b) to (x-1b);
  \draw [<-] (x-3b) to (x-2b);

  \draw [->] (x-2b) to (x-1a);
  \draw [->] (x-3b) to (x-2a);

  \draw [<-] (x0a) to (x0b);
  \draw [->] (x-1a) to (x-1b);
  \draw [->] (x-2a) to (x-2b);
  \draw [->, dashed] (x-3a) to (x-3b);

  \draw [<-] (x2a) to (y-a);
  \draw [<-] (y-a) to (x3a);

  \draw [->] (yb) to (x1b);
  \draw [->] (x0b) to (yb);

  \draw [<-] (x-2a) to (ya);
  \draw [<-] (ya) to (x-3a);

  \draw [->] (y-b) to (x-1b);
  \draw [->] (x0b) to (y-b);

\draw[yshift=-2.3cm]
  node[below,text width=9cm] 
  {
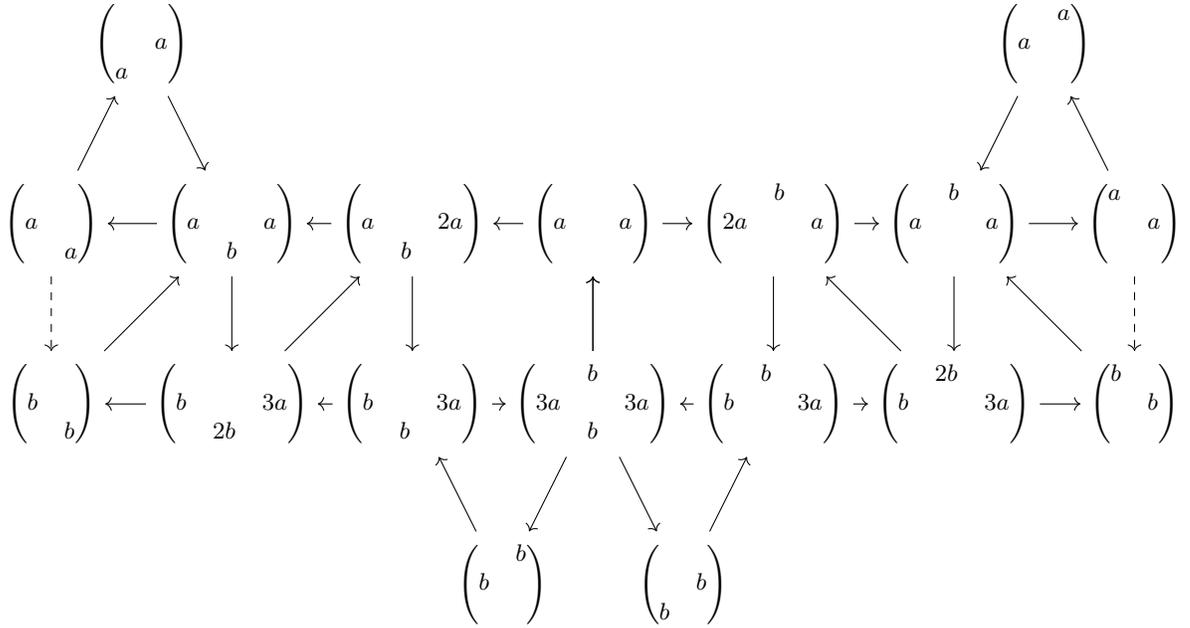
  Figure 18e. The quiver and functions for the cluster algebra on $\Conf_4 \A_{G}$ after the fifth stage of mutation.
  };

\end{tikzpicture}
\end{center}

\begin{center}
\begin{tikzpicture}[scale=2.4]

  \node (x0a) at (0,0) {$\dlr{a}{a}$};
  \node (x1a) at (1,0) {$\tcfu{b}{2a}{a}$};
  \node (x2a) at (2,0) {$\tcfu{b}{a}{a}$};
  \node (x3a) at (3,0) {$\tcfr{a}{}{a}$};
  \node (x-1a) at (-1,0) {$\tcfd{a}{b}{2a}$};
  \node (x-2a) at (-2,0) {$\tcfd{a}{b}{a}$};
  \node (x-3a) at (-3,0) {$\tcfl{}{a}{a}$};

  \node (x0b) at (0,-1) {$\dlr{b}{b}$};
  \node (x1b) at (1,-1) {$\tcfu{b}{b}{3a}$};
  \node (x2b) at (2,-1) {$\tcfu{2b}{b}{3a}$};
  \node (x3b) at (3,-1) {$\tcfr{b}{}{b}$};
  \node (x-1b) at (-1,-1) {$\tcfd{b}{b}{3a}$};
  \node (x-2b) at (-2,-1) {$\tcfd{b}{2b}{3a}$};
  \node (x-3b) at (-3,-1) {$\tcfl{}{b}{b}$};

  \node (ya) at (-2.5,1) {$\tcfr{}{a}{a}$};
  \node (yb) at (-0.5,-2) {$\tcfr{}{b}{b}$};
  \node (y-a) at (2.5,1) {$\tcfl{a}{a}{}$};
  \node (y-b) at (0.5,-2) {$\tcfl{b}{b}{}$};

  \draw [<-] (x1a) to (x0a);
  \draw [<-] (x2a) to (x1a);
  \draw [<-] (x3a) to (x2a);

  \draw [<-] (x1b) to (x0b);
  \draw [<-] (x2b) to (x1b);
  \draw [<-] (x3b) to (x2b);

  \draw [->] (x1b) to (x0a);
  \draw [->] (x2b) to (x1a);
  \draw [->] (x3b) to (x2a);

  \draw [->] (x0a) to (x0b);
  \draw [->] (x1a) to (x1b);
  \draw [->] (x2a) to (x2b);
  \draw [->, dashed] (x3a) to (x3b);

  \draw [<-] (x-1a) to (x0a);
  \draw [<-] (x-2a) to (x-1a);
  \draw [<-] (x-3a) to (x-2a);

  \draw [<-] (x-1b) to (x0b);
  \draw [<-] (x-2b) to (x-1b);
  \draw [<-] (x-3b) to (x-2b);

  \draw [->] (x-1b) to (x0a);
  \draw [->] (x-2b) to (x-1a);
  \draw [->] (x-3b) to (x-2a);

  \draw [->] (x0a) to (x0b);
  \draw [->] (x-1a) to (x-1b);
  \draw [->] (x-2a) to (x-2b);
  \draw [->, dashed] (x-3a) to (x-3b);

  \draw [<-] (x2a) to (y-a);
  \draw [<-] (y-a) to (x3a);

  \draw [->] (yb) to (x0b);
  \draw [->] (x-1b) to (yb);

  \draw [<-] (x-2a) to (ya);
  \draw [<-] (ya) to (x-3a);

  \draw [->] (y-b) to (x0b);
  \draw [->] (x1b) to (y-b);

\draw[yshift=-2.3cm]
  node[below,text width=9cm] 
  {
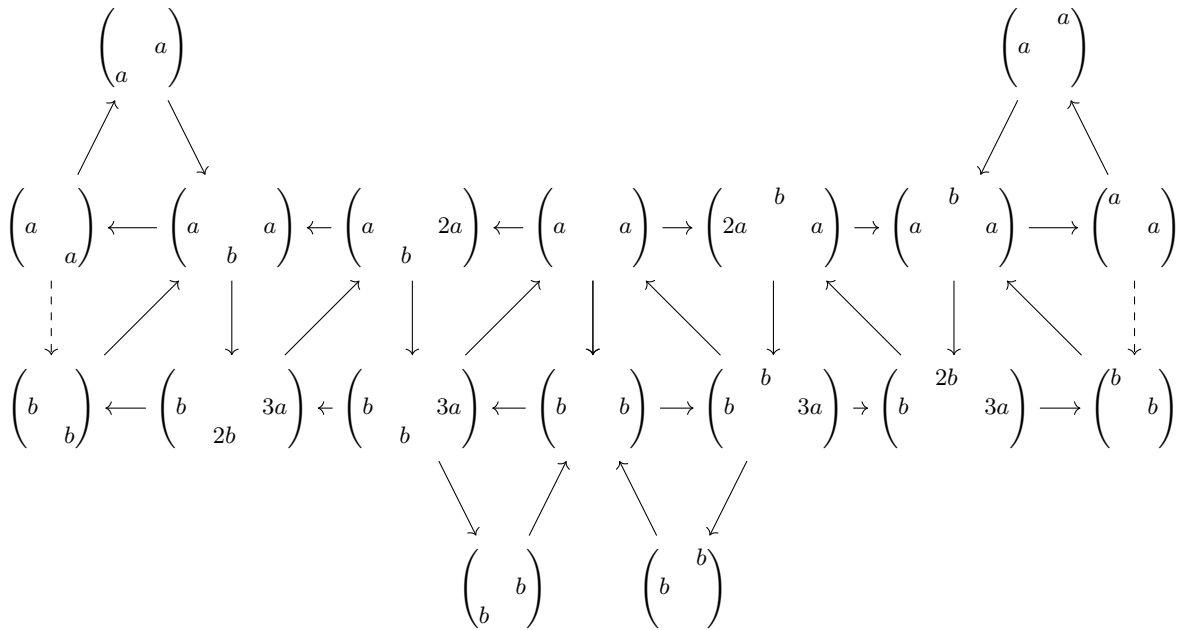
  Figure 18f. The quiver and functions for the cluster algebra on $\Conf_4 \A_{G}$ after the sixth and final stage of mutation.
  };

\end{tikzpicture}
\end{center}

\subsection{The space $\X_{G,S}$}

Let us mention here that once we have constructed the cluster algebra structure on the spaces $\Conf_m \A_G$, it is straightforward to derive the $\X$-variety structure on $\Conf_m \B_G$. The general framework is explained in \cite{FG2}, and some details specific to the case of configurations of principal flags and configurations of flags can be found in \cite{Le}. We give a short summary here which is adapted to the general case. The statements below are conditional upon the conjectures from section 3, and thus hold for types $A, B, C, D$ and $G_2$.

\begin{theorem} $\Conf_m \B_G$ has the structure of a cluster $\X$-variety. Together with the cluster structure that we have constructed on $\Conf_m \A_G$, we obtain a cluster ensemble in the sense of \cite{FG2}.
\end{theorem}

Suppose we have a cluster $\A$-variety with seed $\Sigma=(I, I_0,B, d)$. Then for every non-frozen index $i \in I$, we get a cluster variable $X_i$. We have a map $p: \A_{\Sigma} \to \X_{\Sigma}$ given by 

$$p^*(X_i) = \prod_{j \in I}A_j^{B_{ij}}.$$

The functions $p^*(X_i)$, which a priori live on $\Conf_3 \A_G$, turn out to descend to $\Conf_3 \B_G$. The reason is that the cluster functions $A_j$ that we constructed on $\Conf_3 \A_G$ were invariants of tensor products:

$$A_j \in [V_{\lambda} \otimes V_{\mu} \otimes V_{\nu}]^G.$$

Now recall that $G/U$ has a left action of $H$, the Cartan subgroup. The functions on $G/U$ decompose as a representation of $G$ as
$$\bigoplus_{\lambda \in \Lambda_+} V_{\lambda}.$$
Then we have that $H$ acts on the summand $V_{\lambda}$ via the weight $\lambda$.

Correspondingly, on $\Conf_3 \A_G$ there is an action of $H^3$. On the summand 
$$[V_{\lambda} \otimes V_{\mu} \otimes V_{\nu}]^G,$$
$(h_1, h_2, h_3)$ acts by $\lambda(h_1)\mu(h_2)\nu(h_3).$

Recall that the quiver for our cluster algebra was chosen precisely so that 
\begin{equation} \sum_{j \in I}B_{ij} (\lambda_j, \mu_j, \nu_j) =0
\end{equation}

This forces the action of $H^3$ on the functions $X_i$ to be trivial. Thus the functions $X_i$ naturally live on $\Conf_3 \B_G$.

Let us check that the torus $\X_{\Sigma}$ is birational to $\Conf_3 \B_G$. Since the functions $X_i$ can be viewed as functions on $\Conf_3 \B_G$, we get a map $p': \Conf_3 \B_G \to \X_{\Sigma}$. 

We will now show that the functions $X_i$ parameterize an open set in $\Conf_3 \B_G$. To do this, we will use the parameterization of double Bruhat cells $G^{u,v}$ from \cite{FG3} and \cite{W}, applied to the particular Bruhat cell $G^{w_0,e}$. 

Recall that the functions on $\Conf_3 \A_G$ were associated to a reduced word composition for $w_0$. Suppose the Weyl group for $G$ is generated by $n$ simple reflections $s_i$, $1 \leq i \leq n$. Take a reduced word 

$$w_0=s_{i_1} s_{i_2} s_{i_3} \cdots s_{i_{K-1}} s_{i_K}$$

Recall that if the simple reflection $s_i$ occurs $a_i$ times, then there are $a_i+1$ cluster variables attached to the simple reflection $s_i$. The first and last of these are frozen, and do not correspond to $\X$-variables. Thus we can naturally put the $\X$-variables in bijection with the simple reflections $s_{i_j}$ except those $s_{i_j}$ which are the leftmost occurence of some simple reflection. For simplicity, let us suppose that $s_{i_1}, \dots, s_{i_n}$ are $s_1, \dots, s_n$ is some order. Then we may put the $X$ variables in bijection with $s_{i_j}$ for $n < j \leq K$. Let $X_{j}$ be the $\X$-function attached to the simple reflection $s_{i_j}$ for $n < j \leq K$.

It is known that there is a parameterization of $\Conf_3 \B_G$ given by three flags of the form
$$(B^+ ,u^-B^+, B^-).$$
Here $u^-$ is determined up to conjugation by $H$.

We introduce some notation. Let $b^-$ be an element of $B^-$. There is a natural map 
$$\pi: B^- \rightarrow H=B/[B,B]$$
A choice of opposite flags $B^+$ and $B^-$ gives an inclusion
$$i: H \rightarrow B^-.$$
We set
$$\rho(b^-):=i(\pi(b^-))^{-1}b^-.$$
This associates to each element of $B^-$ an element of $U^-$. We will be interested in $\rho(b^-)$ up to the adjoint action of $H$.

Then the coordinates $X_{j}$ give a parameterization of $u^-$ by the following formula:
$$u^-=\rho(b^-),$$
where
\begin{gather*}
b^-= \big( \prod_{j=n+1}^{K} F_{i_{j}} H_{\omega^{\vee}_{i_{j}}}(X_{j}^{-1}) \big) F_n\dots F_3 F_2 F_1 \\
=F_{i_{n+1}} H_{\omega^{\vee}_{i_{n+1}}}(X_{n+1}^{-1}) F_{i_{n+2}} H_{\omega^{\vee}_{i_{n+2}}}(X_{n+2}^{-1}) \dots \\
F_{i_{K-1}}H_{\omega^{\vee}_{i_{K-1}}}(X_{K-1}^{-1})F_{i_K}H_{\omega^{\vee}_{i_K}}(X_{K}^{-1}) F_n\dots F_3 F_2 F_1.
\end{gather*}

Let us explain the notation above. Here, $F_i$ are the usual generators of $U^-$ associated to the simple roots. $\omega^{\vee}_i$ is the fundamental weight attached to the $i^{th}$ node of the Dynkin diagram for the simply connected form of $G^{\vee}$, the Langlands dual of $G$. The weights for $G^{\vee}$ are the coweights for the adjoint form $G'$ of $G$. Then $H_{\omega^{\vee}_i}$ is the cocharacter attached to this coweight.

In this way, we immediately see that the functions $X_{j}$ parameterize an open set in $\Conf_3 \B_G$.

Next, we then need to check that for any gluing of triangles to get a structure of an $\A$-space on $\Conf_4 \A_G$, the $\X$-coordinates on the edge gluing the two triangles parameterize gluings of configurations in $\Conf_3 \B_G$ to get a configuration in $\Conf_4 \B_G$. In other words, we would like to show that there is an equivalence
$$\Conf_4 \B_G \simeq \Conf_3 \B_G \times H \times \Conf_3 \B_G.$$
Let us examine the $\X$-coordinates associated with the $\A$-space $\Conf_4 \A_G$. We saw above that the face $\X$-coordinates parameterize the two copies of $\Conf_3 \B_G.$ It remains to show that the edge $\X$-coordinates parameterize the space of gluings, which can be parameterized by $H$ (the space of gluings is naturally a torsor for $H$).

Explicitly, we $H$ acts by shearing on $\Conf_4 \B_G$ in the following way:
$$h: (B^+, u^-B^+, B^-, u^+B^-) \rightarrow (B^+, u^-B^+, B^-, hu^+B^-).$$
We can now easily check that the edge $\X$-coordinates form a torsor for $H$. Suppose that the function $A_j$ is the edge function lying in the invariant space 
$$[V_{\omega_j} \otimes V_{\omega_j^{\vee}}]^G.$$
Call $X_j$ the corresponding $\X$-coordinate.

\begin{prop} An element $h \in H$ acts on $X_{j}$ by $\alpha_{j}(h)$, where $\alpha_j$ is the $j$-th simple root of $G'$.
\end{prop}

Note that because $G'$ is adjoint, the simple roots $\alpha_j$ span the weight lattice.

\begin{proof} The proof reduces to a computation. The crux of the computation is Conjecture~\ref{edgeweight}, which can be verified by hand for the group of type $G_2$ using the identification of the weights of the cluster functions given in the previous section. We can lift the configuration 
$$(B^+, u^-B^+, B^-, u^+B^-) \in \Conf_4 \B_G$$
to get the configuration $$(U^+, u^-U^+, w_0U^-, u^+ w_0 U^-) \in \Conf_4 \A_G.$$ Then we lift the action of $h$ to get the map 
$$h: (U^+, u^-U^+, w_0U^-, u^+ w_0 U^-) \rightarrow (U^+, u^-U^+, w_0U^-, hu^+ w_0 U^-).$$
We compute the action of $h$ on the $\X$-variables via the corresponding action on $\A$-variables. If we triangulate, one triangle is unaffected. On the other triangle we compute the $\A$ variable functions on the configuration $(U^+, w_0U^-, hu^+ w_0 U^-) \in \Conf_3 \A_G$. But this is the same as the configuration 
$$(h^{-1}U^+, h^{-1}w_0U^-, u^+ w_0 U^-) \in \Conf_3 \A_G,$$
and Conjecture~\ref{edgeweight} tells us that we have a contribution of $\frac{\alpha_{j}}{2}(h)$ from each of the flags on the edge we are shearing.

A similar theorem was proved for $G$ of type $A, B, C, D$ in \cite{Le}. 

\end{proof}

Thus the edge coordinates on the $\X$ space give the usual ``shear'' coordinates. The usual cutting and gluing arguments allow us to conclude the following:

\begin{theorem} The spaces $\A_{G,S}$ and $\X_{G',S}$ together have the structure of a cluster ensemble when $G$ has type $G_2$.
\end{theorem}

\end{document}